\documentclass[12pt,reqno]{amsart}
\usepackage[dvipsnames]{xcolor}
\usepackage{graphicx}
\usepackage{tikz-cd}
\usepackage{amsmath, amssymb}
\usepackage{mathrsfs}
\usepackage{amsthm}
\usepackage{multicol}
\usepackage{color}
\usepackage{bm}
\usepackage{mathtools}
\usepackage{nameref}
\usepackage{url}
\usepackage[colorlinks=true,linkcolor=blue!60!black,citecolor=teal!80!black,urlcolor=RoyalBlue]{hyperref}
\usepackage[utf8x]{inputenc} 
\usepackage[left=2.7cm,right=2.7cm,top=2.65cm,bottom=2.65cm]{geometry}
\usepackage{enumitem}
\usepackage{accents}
\usepackage{import}

\newtheorem{thm}{Theorem}[section]
\newtheorem{cor}[thm]{Corollary}

\newtheorem{prop}[thm]{Proposition}

\newtheorem{dfn&prop}[thm]{Definition and Proposition}
\newtheorem{dfn&thm}[thm]{Definition and Theorem}

\theoremstyle{definition}
\newtheorem{defn}[thm]{Definition}
\newtheorem{observation}[thm]{Observation}

\newtheorem*{structure}{Structure of the article}

\newtheorem{discussion}[thm]{}

\theoremstyle{remark}

\newtheorem{Claim_num}{Claim}
\newtheorem*{remark}{Remark}

\numberwithin{equation}{section}

\newenvironment{subproof1}{\begin{proof}[Proof of claim 1.]}{%
\end{proof}}
\newenvironment{subproof2}{\begin{proof}[Proof of claim 2.]}{%
\end{proof}}

\def\phi{\varphi}

\def\B{{\mathcal{B}}}
\def\CB{{\mathcal{CB}}}

\def\C{{\mathbb{C}}}
\def\D{{\mathbb{D}}}
\def\N{{\mathbb{N}}}
\def\Z{{\mathbb{Z}}}
\def\R{{\mathbb{R}}}
\def\H{{\mathbb{H}}}
\def\Q{\mathbb {Q}}

\newcommand{\e}{\operatorname{e}}

\newcommand{\id}{\operatorname{id}}
\newcommand{\Ima}{\operatorname{Im}}
\newcommand{\Rea}{\operatorname{Re}}

\newcommand{\Addr}{\operatorname{Addr}}

\newcommand{\T}{\mathcal{T}}

\def\s{{\underline s}}

\newcommand*{\defeq}{\mathrel{\vcenter{\baselineskip0.5ex \lineskiplimit0pt
			\hbox{\scriptsize.}\hbox{\scriptsize.}}}%
	=}

\newcommand{\eqdef}{=\mathrel{\vcenter{\baselineskip0.5ex \lineskiplimit0pt
			\hbox{\scriptsize.}\hbox{\scriptsize.}}}}

\title[Criniferous maps with absorbing Cantor bouquets]{Criniferous entire maps with \\absorbing Cantor bouquets}

\author[L. Pardo-Sim\'{o}n]{Leticia Pardo-Sim\'{o}n}
\address{Department of Mathematics \\ The University of Manchester \\ Manchester \\ M13 9PL \\ United Kingdom \\ 
	 \textsc{\newline \indent 
	   \href{https://orcid.org/0000-0003-4039-5556%
	     }{\includegraphics[width=1em,height=1em]{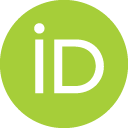} {\normalfont https://orcid.org/0000-0003-4039-5556}}
	       }}
\email{leticia.pardosimon@manchester.ac.uk}
\begin{document}

\newcommand{\new}[1]{\textcolor{magenta}{#1}}
\begin{abstract}
It is known that, for many transcendental entire functions in the Eremenko-Lyubich class $\B$, every escaping point can eventually be connected to infinity by a curve of escaping points. When this is the case, we say that the functions are \emph{criniferous}. In this paper, we extend this result to a new class of maps in $\B$. Furthermore, we show that if a map belongs to this class, then its Julia set contains a \textit{Cantor bouquet}; in other words, it is a subset of $\C$ ambiently homeomorphic to a straight brush.
\end{abstract}

\maketitle
\section{Introduction}
In the study of the behaviour under iteration of an entire function $f$, its \textit{escaping set}
\[I(f)\defeq \{z\in\C : f^n(z)\to \infty \text{ as } n\to \infty\}\]
is of particular importance. In spite of the difference in nature of infinity between the polynomial and transcendental cases, in the former  a super-attracting fixed point  and in the latter an essential singularity, some parallels can be drawn. Following up on Fatou's seminal findings for sine maps \cite{Fatou}, Devaney et al. \cite{devaney_Krych, devaney_tangerman},
observed in the eighties the presence of curves escaping uniformly to infinity under iteration of exponential maps. These curves can be understood as analogues of the so-called \textit{dynamic rays} of polynomials \cite{Orsaynotes}, and are a  powerful combinatorial tool in the study of Julia sets. This led Eremenko to conjecture in \cite{erem89} that any point in the escaping set of a transcendental entire map can be connected to infinity by such an escaping curve. Known as the \emph{strong form of Eremenko's conjecture}, this problem has motivated much research in transcendental dynamics in the last 30 years. Using the terminology introduced in \cite{lasse_dreadlocks}, we refer to entire functions having \textit{dynamic rays} or \textit{(Devaney) hairs} as \textit{criniferous}:
\begin{defn}[Criniferous functions \cite{lasse_dreadlocks}] An entire function $f$ is \emph{criniferous} if the following holds for every $z \in I(f)$: For all sufficiently large $n$, there is an arc $\gamma_n$ connecting $f^n(z)$ to $\infty$, in such a way that $f$ maps $\gamma_n$ injectively onto $\gamma_{n+1}$, and so that $\min_{z\in \gamma_n}\vert z\vert \to\infty$ as $n \to \infty$.
\end{defn}
Despite the existence of counterexamples to the strong form of Eremenko's conjecture, \cite{RRRS, lasse_arclike}, increasingly larger subclasses of the \emph{Eremenko-Lyubich class} $\mathcal{B}$ have been shown to be criniferous. Recall that for an entire function $f$, its \textit{singular set} $S(f)$ is the closure of the set of its finite critical and asymptotic values, and that the widely studied \emph{class $\B$} consists of all transcendental entire maps $f$ with bounded $S(f)$. See \cite{Dave_survey} for a survey of results. After it was shown that all exponential and cosine maps are criniferous \cite{Schleicher_Zimmer_exp, dierkCosine}, the same was proved for certain maps $f\in\B$ of \textit{finite order}, i.e., so that $\log \log \vert f(z)\vert=O(\log \vert z \vert)$ as $\vert z \vert \rightarrow \infty$. Namely, the fact of being criniferous was shown for all finite order $f\in \B$  of \textit{disjoint type}, that is, with connected Fatou set $F(f)$ and so that $P(f)\defeq\overline{\bigcup_{n\in \N}f^n(S(f))}\Subset F(f)$; \cite{Baranski_Trees, RRRS}. Even more generally, it is established in \cite[Sections 4 and 5]{RRRS} that a larger class of functions in $\B$ satisfying certain geometric condition, called \textit{uniform head-start}, are criniferous. In particular, this class includes all finite compositions of functions $f\in \B$ of finite order; see \cite[Theorem 1.2]{RRRS}.

In this paper, we extend this result to a further subclass of $\B$. More specifically, if $f\in \B$, then, it is easy to see that for all $\lambda \in \C^{\ast}$ with $\vert \lambda \vert$ sufficiently small, the function $\lambda f$ is of disjoint type and belongs to the \textit{parameter space} of $f$, that is, $f$ and $\lambda f$ are quasiconformally equivalent; see \S \ref{sec_CB} for definitions. In particular, their dynamics near infinity are related by a certain analogue of Böttcher's Theorem for transcendental maps,~\cite{lasseRigidity}. One might regard disjoint type functions as having the simplest dynamics among those in the parameter space of $f$ and, in particular, they play an analogous role for $f$ as $z\mapsto z^d$ does for a polynomial of degree $d$. By results in \cite{lasseBrushing}, we know that if $\lambda f\in \B$ is of finite order and of disjoint type, then its Julia set $J(\lambda f)$ is a \emph{Cantor bouquet}. This is a topological object introduced by Aarts and Oversteegen \cite{AartsOversteegen} in the study of the dynamics of exponential and sine maps, defined as follows.

\begin{defn}[Cantor bouquet {\cite{AartsOversteegen,lasseBrushing}}] \label{def_brush}
A subset $B$ of $[0,+\infty) \times (\R\setminus\Q)$ is a \emph{straight brush} if the following properties are satisfied:
\begin{itemize}
\item The set $B$ is a closed subset of $\R^2$.
\item For each $(x,y) \in B$, there is $t_y \geq 0$ so that $\{x \colon (x,y)\in B\} = [t_y, +\infty)$. 
The set $[t_y, +\infty) \times \{y\}$ is called the \textit{hair} attached at $y$, and the point $(t_y, y)$ is called its \textit{endpoint}.
\item The set $\{y \colon (x,y) \in B \text{ for some } x\}$ is dense in $\R\setminus\Q$. 
Moreover, for every $(x, y) \in B$, there exist two sequences of hairs attached respectively at $\beta_n, \gamma_n \in \R\setminus\Q$ such that $\beta_n < y < \gamma_n$, $\beta_n, \gamma_n \to y$ and $t_{\beta_n}, t_{\gamma_n} \to t_y$ as $n\to\infty$.
\end{itemize}
\noindent A \emph{Cantor bouquet} is any subset $A\subset \R^2$ that is \textit{ambiently homeomorphic} to a straight brush~$B$, that is, there is a homeomorphism $\psi\colon \R^2 \to \R^2$ such that $\psi(A)=B$. A \textit{hair} (resp. \textit{endpoint}) of a Cantor bouquet is any preimage of a hair (resp. endpoint) of a straight brush under a corresponding ambient homeomorphism.
\end{defn}

We shall show that for $f\in \B$, $J(\lambda f)$ being a Cantor bouquet for small $\vert \lambda \vert$ is enough for $f$ to be criniferous. More precisely, we introduce the following class of functions:
\begin{defn}[The class $\CB$] We say that $f \in \B$ belongs to the \textit{class $\CB$} if $J(\lambda f)$ is a Cantor bouquet for some\footnote{and hence all for which $\lambda f$ is of disjoint type; see Proposition \ref{prop_classCB}.} $\vert \lambda \vert$ sufficiently small. 
\end{defn}

We note that the class $\CB$ in particular includes all entire maps that are a finite composition of functions of finite order in $\B$, and more generally, those satisfying a linear uniform head-start condition, see \S\ref{sec_CB}. Our first result generalizes \cite[Theorem~1.2]{RRRS}.
\begin{thm}\label{thm_intro_rays} All functions in $\CB$ are criniferous. In particular, if $f\in \CB$, then every point $z\in I(f)$ can be connected to $\infty$ by a curve $\gamma$ such that $f^n\vert_{\gamma}\to \infty$ uniformly.
\end{thm}
In fact, Theorem \ref{thm_intro_rays} is a consequence of a more general result: for each $f\in \CB$, we show that its Julia set contains a Cantor bouquet, consisting of escaping curves with their endpoints, and to which all escaping points are eventually mapped. For each $R>0$, we denote 
\begin{equation}\label{eq_JR}
J_R(f) \defeq \lbrace z \in J(f) : |f^n(z)|\geq R \text{ for all } n \geq 1\rbrace.
\end{equation}
\begin{defn}[Absorbing Cantor bouquets] Let $f$ be entire and let $X\subset J(f)$ be a Cantor bouquet. We say that $X$ is \textit{strongly absorbing} if $f(X)\subset X$ and $J_R(f) \subset X$ for some $R>0$. 
\end{defn}
With this terminology, \cite[Theorem 1.6]{lasseBrushing} implies that all functions that are a finite composition of functions in $\B$ of finite order contain strongly absorbing Cantor bouquets. As the main result of this paper, we generalize this result to all maps in $\CB$. 
\begin{thm}\label{thm_intro_CB}
If $f\in \CB$, then for every $Q>0$, $J(f)$ contains a strongly absorbing Cantor bouquet $X\subset J_Q(f)$.
\end{thm} 

The proof of Theorem \ref{thm_intro_CB} relies on the combination of two results: first, by \cite{lasseRigidity}, any $f\in \CB$ is conjugate to the restriction of any disjoint type map $\lambda f$ on $J_R(\lambda f)$ for large $R>0$. Then, we show that for any such $R$, the Cantor bouquet $J(\lambda f)$ can be projected  in a continuous fashion to a strongly absorbing Cantor bouquet contained in $J_R(\lambda f)$. This is achieved exploiting the fact that the topological structure of Cantor bouquets only allows for their curves to ``wind" forwards and backwards in a controlled way; see Figure \ref{figure: proj_cantor}. More precisely, we obtain the following.
\begin{thm}[Continuous projection on absorbing bouquets]\label{thm_intro_pi} Suppose that $f\in \B$ is of disjoint type, and that $J(f)$ is a Cantor bouquet. Then, for each $R>0$, there is a continuous surjective function $\pi_R\colon J(f) \to X$, where $X\subset J_R(f)$ is a strongly absorbing Cantor bouquet. 
\end{thm}
\noindent More details on the definition and properties of the maps $\pi_R$ are provided in Theorem ~\ref{thm_pi_cont}. 

Finally, we note that the results in this paper are used in \cite{mio_splitting, mio_cosine}, where a total description of the Julia set of certain maps in $\CB$ with escaping critical values  introduced in \cite{mio_orbifolds}, is provided. For a unified account of the results, see \cite{mio_thesis}.

\begin{structure}
In \S\ref{sec_log} we recall the \textit{logarithmic change of variables} and prove some basic results we shall require. In \S\ref{sec_Cantor} we study general properties of functions whose Julia set is a {Cantor bouquet}, and prove Theorem \ref{thm_intro_pi}. Finally, \S\ref{sec_CB} gathers our results regarding the class $\CB$, including that this class is closed under iteration, as well as the proofs of Theorems \ref{thm_intro_rays} and \ref{thm_intro_CB}.
\end{structure}

\subsection*{Basic notation}
As introduced throughout this section, the Fatou, Julia and escaping set of an entire function $f$ are denoted by $F(f)$, $J(f)$ and $I(f)$ respectively. The singular set of $f$ is $S(f)$. We denote the complex plane by $\C$, Riemann sphere by $\widehat{\C}$, and $\C^\ast\defeq \C\setminus \{0\}$. A disc of radius $\epsilon$ centred at a point $p$ will be $B_\epsilon(p)$, and $\D_\epsilon$ when $p=0$. We indicate the closure of a domain $U$ by $\overline{U}$, that must be understood to be taken in $\C$. The annulus with radii $a,b \in \R^+$, $a<b$, and the vertical strip between $x=a$ and $x=b$ are respectively denoted by $A(a,b)\defeq\lbrace w\in \C : a< \vert w \vert < b \rbrace$ and $V(a,b)\defeq\lbrace w\in \C : a< \Rea(w) < b \rbrace$.
\subsection*{Acknowledgements}
I am very grateful to my supervisors Lasse Rempe and Dave Sixsmith for their continuous help and advice. I also thank Vasiliki Evdoridou, Daniel Meyer and Phil Rippon for valuable comments, as well as the referee for much helpful and detailed feedback.

\section{Preliminaries: logarithmic coordinates}\label{sec_log}
A commonly used tool for studying functions in the Eremenko-Lyubich class $\B$, is the \emph{logarithmic change of coordinates}, a technique firstly used in this context in \cite[Section ~2]{eremenkoclassB}; see also \cite[Section 5]{Dave_survey}, \cite[Section 2]{RRRS} or \cite[Section 3]{lasse_arclike}. 

\begin{discussion}[Logarithmic transform]\label{dis_def_transform}
For $f\in \B$, fix a Euclidean disk $\D_L\Supset S(f)$ and define the \textit{tracts} $\T_f$ of $f$ as the connected components of $f^{-1}(\C\setminus \overline{\D}_L)$. Let $\H_{\log L}\defeq \exp^{-1}(\C\setminus \overline{\D}_L)$ be the right half-plane containing all points with real part greater than $\log L$, and let $\T_F\defeq \exp^{-1}(f^{-1}(\C\setminus \overline{\D}_L))$. 
Note that each connected component $T$ of $\T_F$ is a simply connected domain whose boundary is homeomorphic to $\R$. Moreover, by the action of the exponential map, both ``ends'' of the boundaries of $T$ have real parts converging to $+\infty$, and both $\T_F$ and $\H_{\log L}$ are invariant under translation by $2\pi i$. Consequently, we can lift $f$ to a map $F: \T_F\to \H_{\log L}$ satisfying 
\begin{equation}\label{eq_commute_log}
\exp\circ F = f\circ\exp,
\end{equation}
and such that $F$ is $2\pi i$-periodic. We call $F$ a \emph{logarithmic transform} of $f$. Moreover, we call each connected component of $\T_F$ a \emph{logarithmic tract} of $F$.
\end{discussion}

\noindent By construction, the following facts, that will be useful for us, also hold:
\begin{discussion}[Properties of logarithmic transforms]\label{dis_prop_transform} Following \ref{dis_def_transform}, we have:
\begin{enumerate}[label=(\arabic*)]
\item \label{item_a_transform}Each tract $T \in \T_F$ is an unbounded Jordan domain that is disjoint from all its $2\pi i\Z$-translates. The restriction $F\vert_{T} \colon T\to \H_{\log L}$ is a conformal isomorphism whose continuous extension to the closure of $T$ in $\widehat{\C}$ satisfies $F(\infty)=\infty$. We denote the inverse of $F\vert_T$ by $F_T^{-1}$. 
\item \label{item_b_transform} The components of $\T_F$ have pairwise disjoint closures and accumulate only at $\infty$; i.e., if $\{z_n\}_n\subset \T_F$ is a sequence of points, all belonging to different components of $\T_F$, then $z_n\rightarrow\infty$ as $n\rightarrow \infty$.
\end{enumerate}
\end{discussion}

By Carathéodory-Torhorst's Theorem, \cite[Theorem 2.1]{pommerenke_boundary}, for each $T\in \T_F$, the function $F\vert_T$ in \ref{dis_prop_transform}\ref{item_a_transform} can be continuously extended to the boundary of $T$. In addition, since $T$ is a Jordan domain, this extension is a homeomorphism, and in particular, $F\vert_T$ extends continuously to a homeomorphism between the closures $\overline{T}$ and 
$\overline{\H}_{\log L}$ (taken in $\mathbb{C}$). Together with property \ref{dis_prop_transform}\ref{item_b_transform}, this implies that $F$ extends continuously to the 
closure $\overline{\mathcal{T}}_{\!\!F}$ of $\mathcal{T}_F$ in $\mathbb{C}$. We then denote
 \[ J(F) \defeq \left\{z\in\overline{\mathcal{T}}_{\!\!F}: F^j(z)\in \overline{\mathcal{T}}_{\!\!F}\text{ for all $j\geq 0$} \right\}. \] 
 
We say that a logarithmic transform $F$ is of \emph{disjoint type} if the boundaries of the tracts of $F$ do not intersect the boundary of $\H_{\log L}$; i.e., if $\overline{\mathcal{T}}_{\!\!F}\subset \H_{\log L}$. In particular, in such case $J(F) =\bigcap _{n\geq 0} F^{-n}(\overline{\mathcal{T}}_{\!\!F}).$

We shall use the following facts about disjoint type maps.
\begin{prop} \label{prop_disjoint}Let $f\in \B$. The following hold:
\begin{itemize}
\item The function $f$ is of disjoint type if and only if there exists a Jordan domain $D \supset S(f)$ such that $\overline{f(D)} \subset \overline{D}$.
\item If $f$ is of disjoint type and $F$ is a logarithmic transform of $f$, then $F$ is of disjoint type and \begin{equation}\label{eq_Juliadisjoint}
\exp(J(F))=J(f).
\end{equation}
\end{itemize}
\end{prop}
\begin{proof}
A proof of the first statement appears in \cite[Proposition 2.8]{helenaSemi}. The second statement follows from the first one; see \cite[Proposition~3.2 and~3.3]{lasse_arclike}.
\end{proof}

We note that \eqref{eq_Juliadisjoint} does not hold for all $f\in \B$. However, using a standard expansion estimate for logarithmic transforms derived from Koebe’s $\frac{1}{4}$-theorem, see \cite[Lemma~ 1]{eremenkoclassB}, for $F\colon \T_F\rightarrow \H_{\log L}$ and all $z \in \T_F$,
\begin{equation}\label{eq_normalized}
\vert F'(z)\vert \geq \frac{1}{4\pi} \left(\Rea F(z) -\log L \right).
\end{equation}
Note that in particular, if $z\in \H_{\log L+8\pi}$, \eqref{eq_normalized} implies that $\vert F'(z)\vert>2$. This allows us to see that a partial inclusion still holds in some cases: 
\begin{observation}\label{rem_juliaLog} Let $F\colon \T_F\rightarrow \H_{\log L}$ be a logarithmic transform for some $f\in \B$. If $X \subset J(F)\cap \H_{\log L+8\pi}$ and $F(X)\subset X$, then $\exp(X) \subset J(f)$.
\end{observation}
\begin{proof}
For the sake of contradiction, let us suppose that there exists $w\in \exp(X)\cap F(f)$. In particular, $w=\exp(z)$ for some $z\in X$, and by \eqref{eq_normalized} together with the assumptions on $X$, $\vert (F^n)'(z)\vert \geq 2$ for all $n\geq 1$. By \cite{eremenkoclassB}, $I(f) \subset J(f)$, and so $w\notin I(f)$. Since by \eqref{eq_commute_log}, $\exp(F(z)) = f(\exp(z))$, it holds that 
$$\vert f'(w)\vert=\frac{\vert F'(z)\vert \cdot\vert f(w)\vert} { \vert w \vert} \geq 2 \frac{\vert f(w)\vert} {\vert w \vert}.$$
Consequently, by this and using the chain rule, for any $n\geq 0$, $$\vert (f^n)'(w) \vert \geq 2^n \vert f^n(w) \vert / \vert w \vert \geq 2^n L / \vert w \vert.$$
Thus, $\vert (f^n)'(w) \vert\rightarrow \infty $ as $n\rightarrow \infty$, which by Marty's theorem \cite[\S3.3]{Schiff_normal} contradicts that $w\in F(f)\setminus I(f)$. 
\end{proof}

Thanks to \ref{dis_prop_transform}\ref{item_a_transform}, we can define symbolic dynamics for logarithmic transforms. 
\begin{defn}[External addresses for logarithmic transforms] \label{def_extaddrLog}
Let $F$ be a logarithmic transform. An \emph{(infinite) external address} is a sequence $\s= T_0 T_1 T_2 \dots$ of logarithmic tracts of $F$. We denote
\[J_{\s}(F)\defeq \{z\in J(F): F^n(z)\in \overline{T}_n \text{ for all $n\geq 0$}\}. \]
Moreover, $\Addr(J(F))$ is the set of external addresses $\s$ such that $J_\s(F) \neq \emptyset$.
\end{defn}

As a consequence of \eqref{eq_normalized}, the following expansion result holds for points with the same external address:
\begin{prop}[Expansion along orbits]\label{prop_contraction} Let $F: \mathcal{T} \rightarrow \H_{\log L}$ be a logarithmic transform of some $f\in \B$ such that $\overline{\mathcal{T}}_{\!\!F}\subset \H_{\log L+8\pi}$. For each $n\geq 0$ and $\s \in \Addr(J(F))$, if $z,w\in J_{\s}(F)$, then
\begin{equation}\label{eq_contraction}
\vert z -w\vert \leq \frac{1}{2^n} \vert F^n(z)-F^n(w)\vert.
\end{equation} 
\end{prop}
\begin{proof}
By the assumption $\overline{\mathcal{T}}_{\!\!F}\subset \H_{\log L+8\pi}$, using \eqref{eq_normalized}, it holds that $\vert F'(z)\vert \geq 2$ for all $z\in \T_F$. For any $T\in \T_F$, let $F_T^{-1}: \H_{\log L} \rightarrow T$ be the inverse branch of $F$ onto the tract $T$. Then,
\begin{equation}\label{eq_inverse_half}
\vert (F_T^{-1})^{'}(v)\vert\leq \frac{1}{2}\quad \text{ for all } \quad v\in \H_{\log L}.
\end{equation}
Moreover, $F$ is by definition of disjoint type, and hence, 
\begin{equation}\label{eq_JFcap}
J(F)=\bigcap _{n\geq 0} F^{-n}(\overline{\mathcal{T}}_{\!\!F}) \subset \T_F\subset \H_{\log L +8\pi}.
\end{equation}
Let $\s=T_0T_1\ldots $ and choose any pair of points $w,z \in J_\s (F)$. Let $\gamma$ be the straight line joining $F(z)$ and $F(w)$. Since $\H_{\log L}$ is a right half-plane, and hence a convex set, $\gamma \subset \H_{\log L}$. In particular, \eqref{eq_inverse_half} holds for all points in $\gamma$. Moreover, since $F_T^{-1}$ is a conformal isomorphism to its image, the curve $F_T^{-1}(\gamma)$ joins $z$ and $w$. Consequently, 
\begin{equation*}
\vert z -w\vert\leq \ell_{eucl}(F_T^{-1}(\gamma))\!=\!\!\int\!\! \left\vert (F_T^{-1})'(\gamma(t))\right\vert\!\left\vert\gamma'(t)\right\vert \!dt \leq \frac{1}{2} \!\int\! \!\vert \gamma'(t)\vert dt= \!\frac{1}{2} \vert F(z)-F(w)\vert.
\end{equation*}
Since $z,w\in J_\s(F)$, for each $k\geq 0$, the points $F^k(z),F^k(w)$ belong to the same tract, and by \eqref{eq_JFcap}, $F^{k+1}(z),F^{k+1}(w)\subset \H_{\log L}$. Hence, we can iteratively apply the same reasoning as before, and \eqref{eq_contraction} follows by induction.
\end{proof}

\section{Cantor bouquets and Julia sets}\label{sec_Cantor}
In this section, we study entire maps whose Julia set is a Cantor bouquet, and prove a more general version of Theorem \ref{thm_intro_pi}. In addition, Proposition~\ref{prop_Ts} gathers general properties of entire functions with Cantor bouquet Julia sets. Recall that we defined Cantor bouquets in Definition \ref{def_brush}. As notes to this definition, we remark the following.

\begin{remark}
Any two straight brushes are ambiently homeomorphic, see \cite{AartsOversteegen}, which in a broad sense 
means that the homeomorphism preserves the ``vertical'' order of the hairs in the brushes. We note that in some texts, the term \emph{(Devaney) hair} has been used to denote what we call \textit{dynamic ray}, see the definition below. However, hairs of Cantor bouquets are not necessarily dynamic rays; see Proposition \ref{prop_Ts}.
\end{remark}

\begin{defn}[Dynamic rays for transcendental maps {\cite[Definition 2.2]{RRRS}}]\label{def_ray}
Let $f$ be a transcendental entire function. A \emph{ray tail} of $f$ is an injective curve $\gamma :[t_0,\infty)\rightarrow I(f)$, with $t_0>0$, such that
\begin{itemize}
\item for each $n\geq 1$, $t \mapsto f^{n}(\gamma(t))$ is injective with $\lim_{t \rightarrow \infty} f^{n}(\gamma(t))=\infty$. 
\item $f^{n}(\gamma(t))\rightarrow \infty$ uniformly in $t$ as $n\rightarrow \infty$.
\end{itemize}
A \emph{dynamic ray} of $f$ is a maximal injective curve $\gamma :(0,\infty)\rightarrow I(f)$ such that the restriction $\gamma_{|[t,\infty)}$ is a ray tail for all $t > 0$. We say that $\gamma$ \emph{lands} at $z$ if $\lim_{t \rightarrow 0^+} \gamma(t)=z$, and we call $z$ the \emph{endpoint} of $\gamma$.
\end{defn}
\begin{remark} With this terminology, an entire function $f$ is criniferous if every $z\in I(f)$ is eventually mapped to a ray tail. That is, for every $z \in I(f)$, there exists a natural number $N \defeq N(z)$ such that $f^n(z)$ belongs to a ray tail for all $n\geq N$.
\end{remark}

A rather simple yet interesting property of Cantor bouquets, which will play a crucial role in our future arguments, is the following: 
\begin{prop}[Jordan curves that hairs intersect at most once]\label{prop_cantorbouquet} Given a Cantor bouquet $X$, for each $R>0$, there exists a bounded simply connected domain $S_R \Supset \D_R $ such that each hair of $X$ intersects $\partial S_R$ at most once.
\end{prop}
\begin{proof}
Let $B$ be a straight brush and let $\psi: \R^2\rightarrow \R^2$ be the ambient homeomorphism in the definition of Cantor bouquet such that $\psi(X)=B$. Fix any $R\geq 0$. Then, since $\psi(\D_R)$ is a bounded set, as $\psi(\overline{\D}_R)$ is the image of a compact set under a continuous function, $\psi(\D_R) \Subset \{(x,y) \in \R^2 : \vert x \vert< Q \text{ and } \vert y \vert< Q\} \eqdef (-Q, Q)^2$ for some $Q\in \Q$; see Figure ~\ref{figure: proj_cantor}. By the choice of $Q$ being a rational number, each hair of the brush $B$ intersects the boundary of $(-Q, Q)^2$ in at most one point. Defining $S_R\defeq \psi^{-1}( (-Q, Q)^2),$ the proposition follows.
\end{proof}

\begin{figure}[htb]
\centering
\begingroup%
  \makeatletter%
  \providecommand\color[2][]{%
    \errmessage{(Inkscape) Color is used for the text in Inkscape, but the package 'color.sty' is not loaded}%
    \renewcommand\color[2][]{}%
  }%
  \providecommand\transparent[1]{%
    \errmessage{(Inkscape) Transparency is used (non-zero) for the text in Inkscape, but the package 'transparent.sty' is not loaded}%
    \renewcommand\transparent[1]{}%
  }%
  \providecommand\rotatebox[2]{#2}%
  \newcommand*\fsize{\dimexpr\f@size pt\relax}%
  \newcommand*\lineheight[1]{\fontsize{\fsize}{#1\fsize}\selectfont}%
  \ifx\svgwidth\undefined%
    \setlength{\unitlength}{354.33070866bp}%
    \ifx\svgscale\undefined%
      \relax%
    \else%
      \setlength{\unitlength}{\unitlength * \real{\svgscale}}%
    \fi%
  \else%
    \setlength{\unitlength}{\svgwidth}%
  \fi%
  \global\let\svgwidth\undefined%
  \global\let\svgscale\undefined%
  \makeatother%
  \begin{picture}(1,0.52)%
    \lineheight{1}%
    \setlength\tabcolsep{0pt}%
    \put(0,0){\includegraphics[width=\unitlength,page=1]{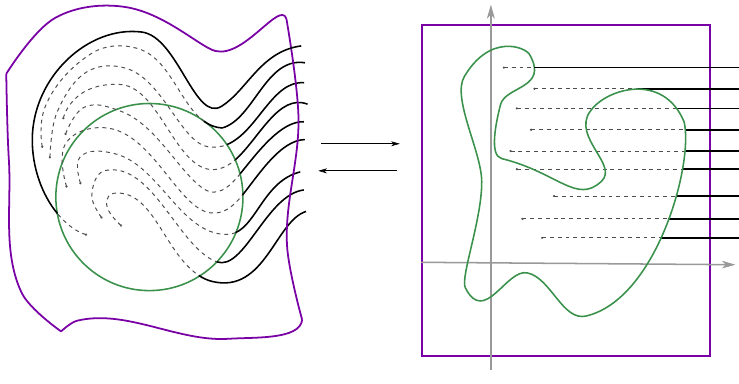}}%
    \put(0.95816779,0.49514881){\color[rgb]{0.47058824,0,0.64705882}\makebox(0,0)[lt]{\lineheight{1.25}\smash{\begin{tabular}[t]{l}$Q$\end{tabular}}}}%
    \put(0.96119159,0.02343454){\color[rgb]{0.47058824,0,0.64705882}\makebox(0,0)[lt]{\lineheight{1.25}\smash{\begin{tabular}[t]{l}$-Q$\end{tabular}}}}%
    \put(0.46302536,0.25702383){\color[rgb]{0,0,0}\makebox(0,0)[lt]{\lineheight{1.25}\smash{\begin{tabular}[t]{l}$\psi^{-1}$\end{tabular}}}}%
    \put(0.47285281,0.33564287){\color[rgb]{0,0,0}\makebox(0,0)[lt]{\lineheight{1.25}\smash{\begin{tabular}[t]{l}$\psi$\end{tabular}}}}%
    \put(0.25324403,0.10507739){\color[rgb]{0.22745098,0.56862745,0.29411765}\makebox(0,0)[lt]{\lineheight{1.25}\smash{\begin{tabular}[t]{l}$\mathbb{D}_R$\end{tabular}}}}%
    \put(0.39611902,0.492125){\color[rgb]{0.47058824,0,0.65098039}\makebox(0,0)[lt]{\lineheight{1.25}\smash{\begin{tabular}[t]{l}$S_R$\end{tabular}}}}%
  \end{picture}%
\endgroup%
\caption{On the left, hairs of a Cantor bouquet intersecting a circle $\partial \D_R$, some of them multiple times. For each hair, dashes represent points with lower potential than that of the last point that intersects $\partial \D_R$. On the right, the image of the hairs to a straight brush under an ambient homeomorphism $\psi$. $[-Q,Q]^2$ is a square whose boundary the hairs intersect at most once, and $S_R\defeq \psi^{-1}((-Q,Q)^2)$.}
\label{figure: proj_cantor}
\end{figure} 
We shall now see that without further assumptions, for an entire function $f$, having a Cantor bouquet Julia set has interesting dynamical implications. Recall that an asymptotic value $a$ of $f$ is \textit{logarithmic} if there exists a neighbourhood $U$ of $a$ and a connected component V of $f^{-1}(U)$ such that $f \colon V \rightarrow U \setminus \{a\}$ is a universal covering map. 

\begin{prop}[Properties of functions with Cantor bouquet Julia sets]\label{prop_Ts} Let $f$ be an entire function and suppose that $J(f)$ is a Cantor bouquet. Then, the following hold:
\begin{enumerate}[label=(\Alph*)]
\item \label{itemA_JCB} $J(f)$ contains neither critical values nor logarithmic asymptotic values. 
\item \label{itemB_JCB} If $\eta$ is a hair of $J(f)$, then $f(\eta)$ is a hair of $J(f)$.
\item \label{itemC_JCB} If $z\in J(f)\cap I(f)$ and $\gamma$ is the (piece of) hair in $J(f)$ joining $z$ to infinity, then $f^n\vert_{\gamma} \rightarrow \infty$ as $n\rightarrow \infty$ uniformly.
\end{enumerate}
If in addition $f$ is of disjoint type, then each hair of $J(f)$ is a dynamic ray together with its endpoint.
\end{prop}

\begin{remark}
Suppose that $f$ is in the so-called \textit{Speiser class}, that is, $f$ is a transcendental entire function with at most finitely many singular values. Then, by Proposition \ref{prop_Ts}\ref{itemA_JCB}, $J(f)$ being a Cantor bouquet implies that $J(f)$ contains
no singular values; see \cite[Lemma~1.2]{goldberg_keen_86}.
\end{remark}
\begin{proof}[Proof of Proposition \ref{prop_Ts}]
We start proving \ref{itemA_JCB} by contradiction. Suppose that $J(f)$ contains a logarithmic asymptotic value $z$, which in particular is in a hair $\eta$. Then, by definition of logarithmic asymptotic value, $f^{-1}(\eta)$ contains a curve separating the plane, which contradicts $J(f)$ being a Cantor bouquet. Now, suppose that there exists a critical point $z\in J(f)$ and let $\eta$ be the hair that contains $z$. By continuity of $f$, $f(\eta)$ must be contained in a hair, say $\tilde{\eta}$. Then, since $f$ acts as the map $w \mapsto w^{\deg(f,z)}$ locally in a neighbourhood of $z$, $J(f)$ being a collection of hairs has the following implications: it can only occur that $\deg(f,z)=2$, $z$ cannot be the endpoint of $\eta$, and $f(z)$ must be the endpoint of $\tilde{\eta}$. For the same reason, if we denote by $e$ and $\tilde{e}$ the respective endpoints of $\eta$ and $\tilde{\eta}$, then $f(e)=\tilde{e}$ and $e$ cannot be a critical point. 

We might assume without loss of generality that $z$ is the critical point in $\eta$ with least potential, i.e., such that the restriction of $\eta$ between $e$ and $z$, that we denote by $\eta[e,z]$, does not contain any other critical point. We can make this assumption because the restriction of $\eta$ between $e$ and any critical point in $\eta$ is compact, and thus, a critical point with minimal potential must exist. We have that $f(\eta[e,z]) \subset \tilde{\eta}$ and $f(e)=f(z)=\tilde{e}$. These two conditions can only be simultaneously fulfilled if $\eta[e,z]$ contains at least another critical point, which contradicts that by the minimality of $z$, all points between $e$ and $z$ are regular. Thus, we have shown that $J(f)$ cannot contain logarithmic asymptotic values nor critical values, and so 
\ref{itemA_JCB} is proved. In particular, we have shown in the proof of \ref{itemA_JCB} that the endpoint of a hair must be mapped to the endpoint of the hair that contains its image under $f$, and so \ref{itemB_JCB} follows.

In order to prove \ref{itemC_JCB}, for each constant $R>0$, let $S_R$ be a bounded set provided by Proposition \ref{prop_cantorbouquet}. Let $z\in J(f)\cap I(f)$ and let $\gamma$ be the (piece of) hair joining $z$ to infinity. Since $z\in I(f)$, there exists $N\defeq N(R,z) \in \N$ such that
\begin{equation} \label{eq_N}
\bigcup_{n\geq N}f^{n}(z)\subset \C \setminus S_R.
\end{equation}
This together with Proposition \ref{prop_cantorbouquet} implies that $f^{n}(\gamma) \subset \C \setminus S_R\subset \C \setminus \D_{R} $ for each $n\geq N$ since, otherwise, $f^{n}(\gamma)$ would intersect $\partial S_R$ at least twice, contradicting Proposition \ref{prop_cantorbouquet}. Hence, since this argument applies to all $R$, $f^n\vert_{\gamma} \rightarrow \infty$ uniformly as $n\rightarrow \infty$. We deduce~\ref{itemC_JCB}.

If $f\in \B$ is of disjoint type, then, by definition, $P(f)\subset F(f)$. By this and since the image of each hair must lie in a hair, the restriction of $f$ to any hair of $J(f)$ is injective. By \cite[Theorem 5.8]{lasse_arclike}, each hair of $J(f)$ but at most its endpoint belongs to $I(f)$, and thus, by \ref{itemC_JCB}, each hair of $J(f)$ is a dynamic ray together with its endpoint. 
\end{proof}

\begin{cor}[Strongly absorbing Cantor bouquet implies criniferous]\label{cor_CB_criniferous} If $f\in \B$ and $J_R(f)$ is a Cantor bouquet for some $R\geq 0$, then $f$ is criniferous.
\end{cor}
\begin{proof}
Let $f\in \B$ as in the statement, and let $z\in I(f)$. Then, since $I(f)\subset J(f)$, \cite{eremenkoclassB}, there exists $N\in \N$ such that $f^N(z)\in J_R(f)$. Then, by Proposition \ref{prop_Ts}\ref{itemC_JCB}, the piece of hair in $J_R(f)$ that joins $f^N(z)$ to infinity is a ray tail, and so $f$ is criniferous. 
\end{proof}

\begin{remark}
The converse of Corollary \ref{cor_CB_criniferous} does not even hold for disjoint type functions. That is, there exist criniferous disjoint type maps whose Julia set is not a Cantor bouquet. Such an example is built in forthcoming work of the author and L. Rempe. 
\end{remark}

Next, we prove Theorem \ref{thm_intro_pi}. It follows from Proposition \ref{prop_Ts} that if $f$ is of disjoint type, then each hair of $J(f)$ contains an unbounded connected component in $J_R(f)$. However, the union of these unbounded components need not be a Cantor bouquet, as we cannot guarantee ``continuity'' across their endpoints: Some hairs, unlike some of their neighbouring ones, might be tangent or intersect $\partial \D_R$ several times; see Figure \ref{figure: proj_cantor}. We overcome this obstacle by looking at points whose orbit never meets a set $S_R$ provided by Proposition~\ref{prop_cantorbouquet}:
\begin{discussion}[Definition of the functions $\pi_{\!_R}$] \label{dis_pi}
Let $f\in \mathcal{B}$ be a disjoint type function for which $J(f)$ is a Cantor bouquet. We define a partial order in $J(f)$ as follows. For $z,w\in J(f)$,
\begin{equation}\label{eq_relation}
w\succeq z  \quad \text{ if and only if } \quad w\in [z,\infty],
\end{equation}
where $[z,\infty]$ is the unique arc in $J(f)\cup \{\infty\}$ connecting $z$ and $\infty$. As usual, we denote by $\succ$ the strict total order associated to $\succeq$, defined as $w\succ z$ if not $z\succeq w$. Moreover, the respective inverse orders $\preceq$ and $\prec$ are defined as $w \preceq z$ if and only if $z \succeq w$, and $w \prec z$ if and only if $z \succ w$.

For each $R>0$, let $S_R$ be a domain given by Proposition \ref{prop_cantorbouquet}, and let $J_R(f)$ be as defined in \eqref{eq_JR}. Then, we define the function $\pi_{\!_R} \colon J(f)\rightarrow J_R(f)$ as 
\begin{equation}\label{eq_pi}
\pi_{\!_R}(z)\defeq \inf\bigg\{w \succeq z : \bigcup_{n\geq 0}f^n(w) \subset \C \setminus S_R \bigg\}.
\end{equation}
\end{discussion}
In other words, provided we show that $\pi_{\!_R}$ is well-defined, the function $\pi_{\!_R}$ acts as the identity for all those points whose forward orbit never meets $S_R$ (and in particular belong to $J_R(f)$), while all other points are projected to the closest point ``to their right'' in the hair they belong to that is of the first type.

The following theorem is a more detailed version of Theorem \ref{thm_intro_pi}.
\begin{thm}[Continuity of $\pi_{\!_R}$]\label{thm_pi_cont}For each $R>0$, the map $\pi_{\!_R}\colon J(f)\rightarrow J_R(f)$ from \eqref{eq_pi} is well-defined and continuous. Moreover, the following hold:
\begin{enumerate}[label=(\alph*)]
\item $X\defeq \pi_{\!_R}(J(f))$ is a strongly absorbing Cantor bouquet; \label{item:Xbouquet}
\item $\pi_{\!_R}\vert_X$ is the identity map;\label{item:pi_id}
\item The set $B\defeq\{ z\in J(f)\colon f(\pi_{\!_R}(z))\neq \pi_{\!_R}(f(z))\}$ is bounded, and if $z\in B$, then  $[\pi_{\!_R}(f(z)),f(\pi_{\!_R}(z))]$ is contained in $\pi(f(\overline{B}))$.\label{item:pi_SR}
\end{enumerate}
If, in addition, there is $L>0$ such that $S(f)\subset \D_L$ and $f^{-1}(\C\setminus \D_L)\subset \C\setminus \D_{e^{8\pi}L},$ then there is a constant $M>R$ such that $\pi_{\!_R}(z) \in \overline{A(M^{-1}\vert z \vert, M \vert z \vert)}$ for all $z\in J(f).$ 
\end{thm}
\begin{proof}
Let us fix some $R>0$. Seeking simplicity of notation, $\pi_{\!_R}$ will be denoted by $\pi$. For each $n\geq 0$, we define the functions $\pi_n: J(f) \rightarrow J(f)$ as
\begin{equation}\label{eq_defpin}
\pi_n(z)\defeq \min\bigg\{w \succeq z : \bigcup^n_{j=0}f^j(w) \subset \C \setminus S_R \bigg\}.
\end{equation}

We are aiming to prove that the functions $\{\pi_n\}_{n\in \N}$ are well-defined, continuous and converge to the function $\pi$ as $n$ tends to infinity. Seeking clarity of exposition, we may assume without loss of generality that $J(f)$ is an embedding in $\C$ of a straight brush. Otherwise, we could prove continuity of the functions $\{\pi_n\}_{n\in \N}$ by showing that, after the usual identification of $\C$ and $\R^2$, the functions $\{(\psi^{-1}\circ \pi_n \circ \psi) \colon A \rightarrow A\}_{n\in \N}$ are continuous, where $A$ is a straight brush and $\psi$ is the corresponding ambient homeomorphism such that $\psi(A)=J(f)$. Thus, with this assumption made, the set $S_R$ from Proposition \ref{prop_cantorbouquet} can be chosen to be of the form $S_R=(-Q,Q)^2$ for some $Q\in \Q$. Then, each hair of $J(f)$ is a horizontal half-line, parallel to the real axis, that by Proposition \ref{prop_cantorbouquet} intersects $\partial S_R$ in at most one point belonging to the vertical line $\{z: \Rea\: z=Q\}$.

For each $n\in \N$ and each hair $\eta$ of $J(f)$, let 
$$z_n\defeq z_n(\eta)=\min\bigg\{z\in \eta: \bigcup^n_{j=0} f^j(z) \subset \C \setminus S_R \bigg\},$$
where the minimum is taken with respect to the relation ``$\preceq$'' from \eqref{eq_relation}. Note that each point $z_n$ is indeed well-defined: by Proposition \ref{prop_Ts}, for each $0\leq j\leq n$, $f^j \vert_{ \eta}$ is a bijection to a hair of $J(f)$, and by Proposition \ref{prop_cantorbouquet}, $f^j(\eta)\cap \partial S_R$ consists of at most one point. Thus, since $\partial S_R \subset \C \setminus S_R$, $z_n$ is either the endpoint of $\eta$, or $z_n$ is the point in $\eta$ with greatest potential such that $f^{j}(z_n) \in \partial S_R$ for some $j\leq n$. This allows us to conveniently express, for each $n\in \N$, the action of $\pi_n$ the following way:
\begin{align*}
\text{ if } z \text{ is in the hair } \eta \subset J(f), \ \text{ then } \ \pi_n(z)& = \renewcommand{\arraystretch}{1.5}\left\{\begin{array}{@{}l@{\quad}l@{}} z_n(\eta) & \text{if } z\prec z_n(\eta), \\
z & \text{if } z\succeq z_n(\eta),
\end{array}\right.\kern-\nulldelimiterspace
\end{align*}
and thus, it follows that the functions $\pi_n$ are well-defined.
\begin{Claim_num}\label{claim1}
For each hair $\eta$ of $J(f)$ and all $n\geq 0$, $z_{n}(\eta)\preceq z_{n+1}(\eta)$.
\end{Claim_num}
\begin{subproof1}
Since $f^n\vert_{\eta}$ maps bijectively to another hair, for any $z,w \in \eta$, $z\prec w$ if and only if $f^n(z)\prec f^n(w).$ In particular, using Proposition \ref{prop_cantorbouquet}, if $f^{n+1}(z_n) \in \C\setminus S_R$, then $z_n=z_{n+1}$. If on the contrary $f^{n+1}(z_n) \in S_R$, then $f^{n+1}(z_{n+1})=f^{n+1}(\eta) \cap \partial S_R$, and thus $z_n\prec z_{n+1}$.
\end{subproof1}

In order to prove continuity of each function $\pi_n$, we will use that for ``close enough'' hairs, their corresponding ``projection points'' $z_n$ are close. More precisely:
\begin{Claim_num}\label{claim2}
Let $\eta$ be a hair of $J(f)$ for which $f^j(z_n(\eta)) \in \partial S_R$ for some $0 \leq j\leq n$. Then, there exists a constant $\alpha\defeq \alpha_n(\eta)>0$ such that if $\tilde{\eta}$ is another hair of $J(f)$ for which $B_\alpha(z_n(\eta))\cap \tilde{\eta}\neq \emptyset,$ then $z_n(\tilde{\eta}) \in B_\alpha(z_n(\eta)).$ 
\end{Claim_num}
\begin{subproof2}
Seeking clarity of exposition, for the proof of this claim we use the notation $z_j\defeq z_j(\eta)$ and $\tilde{z}_j\defeq z_j(\tilde{\eta})$ for each $0 \leq j \leq n$. Let 
\begin{equation}\label{eq_j}
0\leq j_1<j_2< \ldots <j_k\leq n
\end{equation}
be the sequence of iterates for which $f^{j_i}(z_n) \in \partial S_R.$ In particular, by the assumption in the claim, the constant $k$ in \eqref{eq_j} is at least $1$. It follows from Claim \ref{claim1} that $z_j=z_{n}$ for all $j_1 \leq j\leq n.$ Thus, our first goal is to find a neighbourhood $A_1$ of $z_{j_1}$ such that if another hair $\tilde{\eta}$ intersects that neighbourhood, then its corresponding $\tilde{z}_{j_1}$ belongs to it. If $j_1=0$, since $Q$ is a rational number and hairs of straight brushes are by definition at irrational heights, we can choose any $\epsilon_1>0$ such that $\vert Q-\Ima(z_0) \vert< \epsilon_1.$ If $\tilde{\eta}$ is a hair intersecting $B_{\epsilon_1}(z_0)$, since $\tilde{\eta}$ is a horizontal straight line, either $\tilde{z}_{0}$ is the single point in $\partial S_R \cap \tilde{\eta}$ whenever that intersection is non-empty, or $\tilde{z}_{0}$ is the endpoint of $\tilde{\eta}$. In any case $\tilde{z}_{0} \in B_{\epsilon_1}(z_0)$, and so we define $A_1\defeq B_{\epsilon_1}(z_0)$. 

If, on the contrary, $j_1\neq 0$, since by definition $f^j(z_{j_1})\in \C \setminus \overline{S}_R$ for all $j< j_1$, using that $f$ is an open map, we can find $\delta>0$ such that $f^j(B_{\delta}(z_{j_1})) \subset \C \setminus \overline{S}_R$ for all $0\leq j<j_1$. For the same reason as before, we can choose some $\epsilon_1>0$ so that $$B_{\epsilon_1}(f^{j_1}(z_{j_1})) \subset f^{j_1}(B_{\delta}(z_{j_1})) \quad \text{ and } \quad \vert Q- \Ima(f^{j_1}(z_{j_1}))\vert < \epsilon_1. $$
That is, we are choosing some ball centred at $f^{j_1}(z_{j_1})$ and contained in the $j_1$th-image of $B_{\delta}(z_{j_1})$ such that, by our assumption on $S_R$ being the open square $(-Q, Q)^2$ and because $f^{j_1}(z_{j_1}) \in \partial S_R$, half of the ball lies in $S_R$; see Figure \ref{figure_proofBouquet}. Let $A_1$ be the connected component of the $j_1$-th preimage of $B_{\epsilon_1}(f^{j_1}(z_{j_1}))$ that contains $z_n \in B_{\delta}(z_{j_1})$; that is, $$ A_1\subseteq f^{-j_1}(B_{\epsilon_1}(f^{j_1}(z_{j_1}))) \cap B_{\delta}(z_{j_1}).$$

\begin{figure}[htb]
\centering
\begingroup%
  \makeatletter%
  \providecommand\color[2][]{%
    \errmessage{(Inkscape) Color is used for the text in Inkscape, but the package 'color.sty' is not loaded}%
    \renewcommand\color[2][]{}%
  }%
  \providecommand\transparent[1]{%
    \errmessage{(Inkscape) Transparency is used (non-zero) for the text in Inkscape, but the package 'transparent.sty' is not loaded}%
    \renewcommand\transparent[1]{}%
  }%
  \providecommand\rotatebox[2]{#2}%
  \newcommand*\fsize{\dimexpr\f@size pt\relax}%
  \newcommand*\lineheight[1]{\fontsize{\fsize}{#1\fsize}\selectfont}%
  \ifx\svgwidth\undefined%
    \setlength{\unitlength}{396.8503937bp}%
    \ifx\svgscale\undefined%
      \relax%
    \else%
      \setlength{\unitlength}{\unitlength * \real{\svgscale}}%
    \fi%
  \else%
    \setlength{\unitlength}{\svgwidth}%
  \fi%
  \global\let\svgwidth\undefined%
  \global\let\svgscale\undefined%
  \makeatother%
  \begin{picture}(1,0.64285714)%
    \lineheight{1}%
    \setlength\tabcolsep{0pt}%
    \put(0,0){\includegraphics[width=\unitlength,page=1]{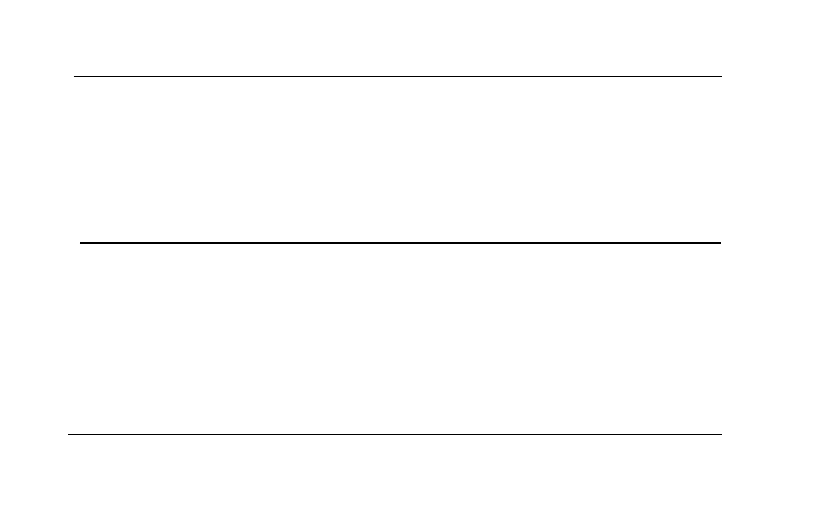}}%
    \put(0.88287152,0.1114353){\color[rgb]{0,0,0}\makebox(0,0)[lt]{\lineheight{1.25}\smash{\begin{tabular}[t]{l}$\eta$\end{tabular}}}}%
    \put(0.87718969,0.34462796){\color[rgb]{0,0,0}\makebox(0,0)[lt]{\lineheight{1.25}\smash{\begin{tabular}[t]{l}$f^{j_1}(\eta)$\\\end{tabular}}}}%
    \put(0.8811142,0.54783121){\color[rgb]{0,0,0}\makebox(0,0)[lt]{\lineheight{1.25}\smash{\begin{tabular}[t]{l}$f^{j_2}(\eta)$\\\end{tabular}}}}%
    \put(0,0){\includegraphics[width=\unitlength,page=2]{pibouquet.pdf}}%
    \put(0.71509765,0.09797617){\color[rgb]{0,0,0}\makebox(0,0)[lt]{\lineheight{1.25}\smash{\begin{tabular}[t]{l}$z_n$\\\end{tabular}}}}%
    \put(0.00610894,0.05352906){\color[rgb]{0,0,0}\makebox(0,0)[lt]{\lineheight{1.25}\smash{\begin{tabular}[t]{l}$x = -Q$\end{tabular}}}}%
    \put(0.01293903,0.61149557){\color[rgb]{0,0,0}\makebox(0,0)[lt]{\lineheight{1.25}\smash{\begin{tabular}[t]{l}$x = Q$\\\end{tabular}}}}%
    \put(0.12912595,0.01118287){\color[rgb]{0,0,0}\makebox(0,0)[lt]{\lineheight{1.25}\smash{\begin{tabular}[t]{l}$y = Q$\end{tabular}}}}%
    \put(0,0){\includegraphics[width=\unitlength,page=3]{pibouquet.pdf}}%
    \put(0.82109015,0.04810849){\color[rgb]{0.41568627,0.23137255,0.91372549}\makebox(0,0)[lt]{\lineheight{1.25}\smash{\begin{tabular}[t]{l}$B_{\delta}(z_n)$\end{tabular}}}}%
    \put(0,0){\includegraphics[width=\unitlength,page=4]{pibouquet.pdf}}%
    \put(0.58722937,0.22159504){\color[rgb]{0.41568627,0.23137255,0.91372549}\makebox(0,0)[lt]{\lineheight{1.25}\smash{\begin{tabular}[t]{l}$f$\end{tabular}}}}%
    \put(0.39631953,0.26313142){\color[rgb]{0.41568627,0.23137255,0.91372549}\makebox(0,0)[lt]{\lineheight{1.25}\smash{\begin{tabular}[t]{l}$f$\end{tabular}}}}%
    \put(0.2871021,0.30450072){\color[rgb]{0.41568627,0.23137255,0.91372549}\makebox(0,0)[lt]{\lineheight{1.25}\smash{\begin{tabular}[t]{l}$f$\end{tabular}}}}%
    \put(0.77942472,0.06648334){\color[rgb]{0.18039216,0.62745098,0.85490196}\makebox(0,0)[lt]{\lineheight{1.25}\smash{\begin{tabular}[t]{l}$A_1$\end{tabular}}}}%
    \put(0,0){\includegraphics[width=\unitlength,page=5]{pibouquet.pdf}}%
    \put(-0.05343832,0.3036762){\color[rgb]{0.18039216,0.62745098,0.85490196}\makebox(0,0)[lt]{\lineheight{1.25}\smash{\begin{tabular}[t]{l}$B_{\epsilon_1}(f^{j_1}(z_n))$\end{tabular}}}}%
    \put(0,0){\includegraphics[width=\unitlength,page=6]{pibouquet.pdf}}%
    \put(0.17868569,0.19219789){\color[rgb]{0.18039216,0.62745098,0.85490196}\makebox(0,0)[lt]{\lineheight{1.25}\smash{\begin{tabular}[t]{l}$f^{-1}$\end{tabular}}}}%
    \put(0.35188005,0.13573514){\color[rgb]{0.18039216,0.62745098,0.85490196}\makebox(0,0)[lt]{\lineheight{1.25}\smash{\begin{tabular}[t]{l}$f^{-1}$\end{tabular}}}}%
    \put(0,0){\includegraphics[width=\unitlength,page=7]{pibouquet.pdf}}%
    \put(0.39062429,0.40374421){\color[rgb]{0.18039216,0.62745098,0.85490196}\makebox(0,0)[lt]{\lineheight{1.25}\smash{\begin{tabular}[t]{l}$f$\end{tabular}}}}%
    \put(0,0){\includegraphics[width=\unitlength,page=8]{pibouquet.pdf}}%
    \put(0.58433713,0.52138926){\color[rgb]{0.18039216,0.62745098,0.85490196}\makebox(0,0)[lt]{\lineheight{1.25}\smash{\begin{tabular}[t]{l}$f$\end{tabular}}}}%
    \put(0.36227613,0.53225592){\color[rgb]{0.18039216,0.62745098,0.85490196}\makebox(0,0)[lt]{\lineheight{1.25}\smash{\begin{tabular}[t]{l}$f$\end{tabular}}}}%
    \put(-0.00652241,0.56039083){\color[rgb]{0.18039216,0.62745098,0.85490196}\makebox(0,0)[lt]{\lineheight{1.25}\smash{\begin{tabular}[t]{l}$f^{j_2}(A_1)$\end{tabular}}}}%
    \put(0,0){\includegraphics[width=\unitlength,page=9]{pibouquet.pdf}}%
    \put(0.18033589,0.58688543){\color[rgb]{0.85098039,0.44313725,0.63921569}\makebox(0,0)[lt]{\lineheight{1.25}\smash{\begin{tabular}[t]{l}$B_{\epsilon_2}(f^{j_2}(z_n))$\end{tabular}}}}%
    \put(0,0){\includegraphics[width=\unitlength,page=10]{pibouquet.pdf}}%
    \put(0.28959446,0.46705511){\color[rgb]{0.85098039,0.44313725,0.63921569}\makebox(0,0)[lt]{\lineheight{1.25}\smash{\begin{tabular}[t]{l}$f^{-1}$\end{tabular}}}}%
    \put(0.57704397,0.4509354){\color[rgb]{0.85098039,0.44313725,0.63921569}\makebox(0,0)[lt]{\lineheight{1.25}\smash{\begin{tabular}[t]{l}$f^{-1}$\end{tabular}}}}%
    \put(-0.01277359,0.23814276){\color[rgb]{0.41568627,0.23137255,0.91372549}\makebox(0,0)[lt]{\lineheight{1.25}\smash{\begin{tabular}[t]{l}$f^{j_1}(B_{\delta}(z_n))$\end{tabular}}}}%
    \put(0.74682574,0.15086466){\color[rgb]{0.85098039,0.44313725,0.63921569}\makebox(0,0)[lt]{\lineheight{1.25}\smash{\begin{tabular}[t]{l}$A_2$\end{tabular}}}}%
    \put(0,0){\includegraphics[width=\unitlength,page=11]{pibouquet.pdf}}%
    \put(0.52771944,0.10561494){\color[rgb]{0.18039216,0.62745098,0.85490196}\makebox(0,0)[lt]{\lineheight{1.25}\smash{\begin{tabular}[t]{l}$f^{-1}$\end{tabular}}}}%
    \put(0,0){\includegraphics[width=\unitlength,page=12]{pibouquet.pdf}}%
    \put(0.68903416,0.1362984){\color[rgb]{0.5372549,0.62745098,0.17254902}\makebox(0,0)[lt]{\lineheight{1.25}\smash{\begin{tabular}[t]{l}$\tilde{z}_n$\\\end{tabular}}}}%
    \put(0,0){\includegraphics[width=\unitlength,page=13]{pibouquet.pdf}}%
    \put(0.86484691,0.13843834){\color[rgb]{0.5372549,0.62745098,0.17254902}\makebox(0,0)[lt]{\lineheight{1.25}\smash{\begin{tabular}[t]{l}$\tilde{\eta}$\\\end{tabular}}}}%
  \end{picture}%
\endgroup%
\caption{Construction of a neighbourhood of $z_n(\eta)$ in Claim \ref{claim2} by pulling back balls centred at $f^j(z_n)$ for all $1\leq j\leq n$ such that $f^j(z_n) \in \partial S_R$.}
\label{figure_proofBouquet}
\end{figure}

We claim that if $\tilde{\eta}$ is any other hair in $J(f)$ such that $\tilde{\eta} \cap A_1 \neq \emptyset,$ then $\tilde{z}_{j_1} \in A_1$. Indeed, by definition of $A_1$, the hair $f^{j_1}(\tilde{\eta})$ must intersect $B_{\epsilon_1}(f^{j_1}(z_{j_1}))$, and so, two cases can occur according to whether $f^{j_1}(\tilde{\eta})$ intersects $S_R$ or not. If $f^{j_1}(\tilde{\eta}) \cap S_R=\emptyset$, by the assumption of $J(f)$ being a straight brush, the endpoint of the hair $f^{j_1}(\tilde{\eta})$ must be contained in $B_{\epsilon_1}(f^{j_1}(z_{j_1})) \setminus S_R$. Thus, $A_1$ contains the endpoint of $\tilde{\eta}$, that we denote by $\tilde{e}$, and so by the choice of $\delta$, $f^j(\tilde{e})\subset \C \setminus S_R$ for all $j< j_1$. Hence, $\tilde{z}_{j_1}=\tilde{e}$ is the endpoint of $\tilde{\eta}$. If, on the other hand, $f^{j_1}(\tilde{\eta})\cap S_R \neq \emptyset$, since $B_{\epsilon_1}(f^{j_1}(z_{j_1}))$ is convex and $f^{j_1}(\tilde{\eta})$ is a horizontal straight line, the intersection $B_{\epsilon_1}(f^{j_1}(z_{j_1})) \cap f^{j_1}(\tilde{\eta}) \cap \partial S_R$ consists of a unique point, that we denote by $p$. In particular, by definition of $\tilde{z}_{j_1}$, $p\preceq f^{j_1}(\tilde{z}_{j_1})$, and so, if $q$ is the preimage of $p$ in $A_1$, $q\preceq \tilde{z}_{j_1}$ by Claim \ref{claim1}. But by the choice of $\delta$, $f^{j}(q) \in \C \setminus S_R$ for all $1\leq j<j_1$, and so, by minimality of $\tilde{z}_{j_1}$, it must occur that $\tilde{z}_{j_1}=q\in A_1$.

If $j_1=n$, we have proved the claim. Otherwise, we can assume that $\epsilon_1$ has been chosen small enough so that for all $j_1< j\leq \min(j_2-1,n)$, $f^{j}(B_{\epsilon_1}(z_n)) \subset \C \setminus \overline{S}_R$. This implies that for any such $j$ and any hair $\tilde{\eta}$ intersecting $A_1$, $f^j(\tilde{z}_{j_1}) \in \C \setminus \overline{S}_R$, and thus, $\tilde{z}_{j_1}=\ldots =\tilde{z}_{\min(j_2-1,n)}$.

If $k=1$ in \eqref{eq_j} we are done. Otherwise, we aim to define a neighbourhood $A_2\subset A_1$ of $z_{j_2}(=z_n)$ with analogous properties to those of $A_1$. That is, such that if $\tilde{\eta}\cap A_2\neq \emptyset$ for some hair $\tilde{\eta}$, then $\tilde{z}_{j_2}\in A_2$. In order to do so, by the same argument as when we chose $\epsilon_1$, we can choose $\epsilon_2$ such that 
$$B_{\epsilon_2}(f^{j_2}(z_{j_2})) \subset f^{j_2}(A_1) \quad \text{ and } \quad \vert Q- \Ima(f^{j_2}(z_{j_2}))\vert < \epsilon_2. $$
Let $A_2$ be the connected component of $f^{-j_2}(B_{\epsilon_2}(f^{j_2}(z_n))) \cap A_1$ that contains $z_n$, and let $\tilde{\eta}$ be any other hair in $J(f)$ such that $\tilde{\eta} \cap A_2 \neq \emptyset.$ By definition, since $f^{j_2}(\tilde{\eta})$ is a horizontal straight line, if $f^{j_2}(\tilde{z}_{j_2-1})\in \C\setminus S_R$, then $\tilde{z}_{j_2-1}=\tilde{z}_{j_2}\in A_2.$ Otherwise, $f^{j_2}(\tilde{\eta})$ intersects $\partial S_R$ in a single point $p\in B_{\epsilon_2}(f^{j_2}(z_n))$ such that $p\succeq f^{j_2}(\tilde{z}_{j_2-1})$, and so $\tilde{z}_{j_2}$ is the preimage of $p$ in $A_2\subset A_1$. If $j_2\neq n$, then choose $\epsilon_2$ small enough so that $f^{j}(B_{\epsilon_2}(z_n)) \subset \C \setminus \overline{S}_R$ for all $j_2\leq j<\min(j_3-1,n)$ and hence, for any $\tilde{\eta}$ intersecting $A_2$ and all $j_1< j\leq \min(j_3-1,n)$, $\tilde{z}_{j_2}=\ldots =\tilde{z}_{\min(j_3-1,n)}$. 

Continuing the process for each $j_i$ in \eqref{eq_j}, we build a nested sequence of open sets $A_k \subseteq \cdots \subseteq A_i\subset \cdots \subseteq A_1$ such that $\tilde{z}_n \in A_k \subset B_{\delta}(z_n)$ for any hair $\tilde{\eta}$ such that $\tilde{\eta} \cap A_k \neq \emptyset.$ Hence, choosing any $\alpha$ so that $B_{\alpha}(z_n) \subset A_k$, the claim follows.
\end{subproof2}

Continuity of the functions $\{\pi_n\}_{n\in \N}$ is now a consequence of Claim \ref{claim2}: let us fix $n\in \N$, let $z\in J(f)$, in particular belonging to some hair $\eta$, and fix any $\epsilon>0$. We want to see that there exists $\delta>0$ such that $\pi_n(B_\delta(z)) \subset B_\epsilon(\pi_n(z))$. One of the following must hold:
\begin{itemize}[noitemsep,wide=0pt, leftmargin=\dimexpr\labelwidth + 2\labelsep\relax]
\item $z \succ z_n(\eta)$, and so $\pi_n(z)=z$. In particular, $f^j(z) \in \C \setminus \overline{S}_R$ for all $j\leq n$, and thus, we can choose $\delta<\epsilon$ small enough such that $f^j(B_\delta(z)) \subset \C \setminus \overline{S}_R$ for all $0\leq j\leq n$. This already implies that $\pi_n(w)=w$ for all $w\in B_\delta(z)\subset B_\epsilon(z)$, since by Proposition \ref{prop_cantorbouquet}, if $\tilde{\eta}$ is the hair containing $w$, then $w\prec z_n(\tilde{\eta})$ only if $f^{j}(w) \in S_R$ for some $j\leq n$, which cannot occur by the choice of $\delta$.

\item $z\prec z_n(\eta)$, and hence $\pi_n(z)=z_n(\eta)$ and $z$ is not the endpoint of $\eta$. Let $\alpha\defeq\alpha_n(\eta)$ be the constant given by Claim \ref{claim2}. If we choose $\delta<\min(\alpha, \epsilon)$, by $J(f)$ being a straight brush, any hair intersecting $B_\delta(z)$ also intersects $B_{\alpha}(z_n(\eta))$. By Claim \ref{claim2}, if $w\in \tilde{\eta} \cap B_\delta(z)$ for some hair $\tilde{\eta}$ and $w\prec z_n(\tilde{\eta})$, then $\pi_n(w)=z_n(\tilde{\eta})\in B_{\min(\alpha, \epsilon)}(z_n(\eta))$. Otherwise, the case of $w\succeq z_n(\tilde{\eta})$, and so $\pi_n(w)=w$, can only occur when $w\in B_\delta(z)\cap B_{\min(\alpha, \epsilon)}(z_n(\eta))$. Thus, we have again that $\pi_n(w) \in B_{\epsilon}(z_n(\eta))$, as we wanted to show.

\item $z= z_n(\eta)$, and hence $\pi_n(z)=z.$ If $f^j(z_n(\eta))\in \C \setminus\overline{S}_R$ for all $j\leq n$, then $z$ is the endpoint of $\eta$, and the same argument as in the first case applies. If, on the contrary, $f^j(z_n(\eta))\in \partial{S}_R$ for some $j\leq n$, then the same argument as in the second case applies.
\end{itemize}	

Now that we have shown that the functions $\{\pi_n\}_{n\in \N}$ are continuous, our next goal is to prove that they converge to a limit function. We start by showing that for each hair $\eta$ of $J(f)$, the sequence $\{z_n(\eta)\}_{n\in \N}$ is convergent by being a Cauchy sequence. By Claim \ref{claim1}, for each $n\geq 1$, either $z_{n-1}=z_n$ or $z_{n-1} \prec z_{n}$, the latter case occurring only if $f^{n}(z_{n-1}) \in S_R\cap \{z: 0\leq\Rea(z)<Q\}$ and $f^{n}(z_{n}) \in \partial S_R$. Thus, in the latter case, the Euclidean length of the piece of the hair $f^n(\eta)$ joining $f^{n}(z_{n-1})$ and $f^{n}(z_{n})$, that we denote by $\gamma$, is at most $Q$. That is, 
\begin{equation} \label{eq_lengthQ}
\ell_{\text{eucl}}(\gamma)\leq Q.
\end{equation}
Since the map $f$ is of disjoint type, and in particular hyperbolic, we can find an open neighbourhood $U$ of $P(f)$ such that $f(U)\subset \overline{U}$; see \cite[Proposition 2.8]{helenaSemi}. We might assume without loss of generality that $U$ has finite Euclidean perimeter and a smooth boundary, since otherwise, we can take a slightly smaller domain $P(f)\subset U'\subset U$ with such properties. Let $W\defeq\C\setminus \overline{U}$ and define the set of tracts $\T_f$ as the connected components of $f^{-1}(W)$. In particular, $\T_f\subseteq f^{-1}(W)\subset W$. Thus, $f \colon \T_f \rightarrow W$ is a covering map, and 
\begin{equation}\label{eq_f1W}
J(f)= \bigcap^\infty_{n=1}f^{-n}(W)\Subset W.
\end{equation}
We can endow $W$ with the hyperbolic metric induced from its universal covering map. Then, it is well-known that $f$ expands uniformly the hyperbolic metric of $W$, that is, there exists a constant $\Lambda>1$ such that $\vert \vert \operatorname{D}f(z)\vert \vert_W\geq \Lambda$ for all $z \in \T_f$; see, for example, \cite[Lemma 5.1]{lasseRigidity}.

By \cite[Lemma 2.1]{lasse_dreadlocks}, only finitely many pieces of tracts in $\T_f$ intersect $\overline{S}_R$. Let us consider the collection 
\begin{equation}\label{eq_collection_Ki}
\{K_1, \ldots, K_{\tilde{M}}\}
\end{equation}
of the closures of such pieces of tracts. By the choice of $\partial U$ being smooth and analytic, so are the boundaries of the tracts in $\T_f$, and, in particular, the boundaries of the sets $\{K_i\}^{\tilde{M}}_{i=1}$. By this, since $K_i \Subset W$ for each $i\leq \tilde{M}$, and the density function $\rho_W$ is continuous in each compact set $\overline{K}_i$, it attains a maximum value $M_{i}$ on it. Let $M\defeq\max_i M_i$. Recall that the straight line $\gamma$ joining $f^{n}(z_{n-1})$ and $f^{n}(z_{n})$ belongs to $J(f) \cap \overline{S}_R$, since in particular is a piece of hair. Hence, by \eqref{eq_f1W}, $\gamma$ must be totally contained in one of the compact sets $\{K_i\}_i$, and by \eqref{eq_lengthQ}, we have the following bound for its hyperbolic length:
\begin{equation}\label{eq_ intmu}
\ell_W(\gamma)=\int \vert \gamma'(t)\vert \rho_{W}(\gamma(t))dt \leq M \int \vert \gamma'(t)\vert dt= M\cdot\ell_{\text{eucl}}(\gamma) \leq M\cdot Q. 
\end{equation}

Let $\beta$ be the piece of the hair $\eta$ joining $z_{n-1}$ and $z_n$. Then, by \eqref{eq_f1W} and since $\vert \vert Df(z)\vert \vert_W\geq \Lambda$ for all $z\in \beta$, 
\begin{equation}\label{eq_pik} 
d_W(z_{n-1},z_n)\leq \ell_W (\beta) \leq \frac{\ell_W(\gamma)}{\Lambda^{n}} \leq \frac{M \cdot Q}{\Lambda^{n}}.
\end{equation}
Note that the upper bounds on the lengths of $\gamma$ and $\beta$ do not depend on the points $z_{n-1}$ and $z_n$, but only on them belonging to $S_R$, and hence $\{z_n(\eta)\}_{n\in \N}$ forms a Cauchy sequence in the complete space $(W, \rho_W)$. This sequence converges to a limit point, that we denote by $z_\eta$. Consequently, we have shown that the functions $\{\pi_n\}_{n\in \N}$ converge to a continuous limit function $\pi$ such that
\begin{align}\label{eq_express_pi}
\text{ if } z\in \eta, \quad \text{ then } \quad \pi(z)& = \renewcommand{\arraystretch}{1.5}\left\{\begin{array}{@{}l@{\quad}l@{}} z_\eta & \text{if } z\prec z_\eta, \\
z & \text{if } z\succeq z_\eta.
\end{array}\right.\kern-\nulldelimiterspace
\end{align}
In particular, using the definition of the functions $\pi_n$ in \eqref{eq_defpin}, the limit function $\pi$ must be equal to that defined in \eqref{eq_pi}. Note also that for all $z\in J(f)$, $\pi(z)\in \C \setminus S_R \subset \C \setminus \D_R$, and so the codomain of $\pi$ is indeed $J_R(f)$. We have then shown that $\pi$ is well-defined and continuous. 

Next, let $X\defeq \pi(J(f))$. Since $X$ is the continuous image of a Cantor bouquet under a function that maps each hair to an unbounded curve of itself, $X$ is a Cantor bouquet. By \eqref{eq_express_pi}, $\pi\vert_X$ is the identity map, and if $Q>R$ is a constant such that $S_R \Subset \D_Q$, then for any $w\in J_Q(f)$, $\pi(w)=w$, and so $J_Q(f) \subset X$. We have shown \ref{item:Xbouquet} and \ref{item:pi_id}.

Let $\eta$ be a hair of $J(f)$ and note that by Proposition \ref{prop_Ts}, $f\colon \eta\to \tilde{\eta}$ is a bijection for some hair $\tilde{\eta}$. By this and the definition of $z_{\tilde{\eta}}$ as the limit of $\{z_n(\tilde{\eta})\}_{n\in \N}$, exactly one of the following holds:
\begin{enumerate}[label=(\arabic*)]
\item  both $z_\eta,z_{\tilde{\eta}} \in \C \setminus S_R$ and $f(z_\eta)=z_{\tilde{\eta}}$; 
\item $z_\eta \in \partial S_R$ and $z_{\tilde{\eta}} \prec f(z_\eta)$. \label{item_2inproof}
\end{enumerate}
Moreover, in the latter case, by Proposition \ref{prop_cantorbouquet}, we have that $[z_{\tilde{\eta}},f(z_{\eta})]\subset f(\overline{S_R})$. Since $\pi\vert_X\equiv \id$, $f(\pi(z))\neq \pi(f(z))$ can only occur for $z\in S_R\cap \eta$ such that \ref{item_2inproof} holds. Thus, \ref{item:pi_SR} follows. 

Finally, for the last part of the statement, we have the additional assumption that there exists $L>0$ such that $S(f)\subset \D_L$ and $f^{-1}(\C\setminus \D_L)\subset \C\setminus \D_{e^{8\pi}L},$ and we aim to get for each $z\in J(f)$ an estimate on the Euclidean distance between $z$ and $\pi(z)$. Fix any $z\in J(f)$ and let $\eta$ be the hair it belongs to. If $z\succeq z_\eta$, then $\pi(z)=z$. Otherwise, by \eqref{eq_express_pi}, $\pi(z)=z_\eta$ and $z_\eta$ cannot be the endpoint of $\eta$. Thus, it must occur that $f^n(z_\eta)\in \partial S_R$ for some $n\geq 0$, and in particular,
\begin{equation}\label{eq_fin_projC}
f^n(z)\in (S_R\setminus\D_{L\e^{8\pi}}) \cap \{z: 0\leq\Rea(z)<Q\}
\end{equation}
by $J(f)$ being a straight brush. Let $F\colon\mathcal{T}_F \rightarrow \H_{\log L}$ be a logarithmic transform of $f$ such that $\mathcal{T}_F\subset \H_{\log L+8\pi}$, and let $w$ and $w_\eta$ be a pair of respective preimages of $z$ and $z_\eta$ under the exponential map, lying in the same connected component of $J(F)$. In particular, each hair of $J(f)$ must be lifted to a connected component of $J(F)$, and thus, $w,w_\eta \in J_\s(F)$ for some $\s \in \Addr(J(F))$. Moreover, recall from \eqref{eq_collection_Ki} that only finitely many pieces of tracts $\T_f$ intersect $S_R$, and so, $f^n(z)$ and $f^n(z_\eta)$ must belong to one of the compact sets in \eqref{eq_collection_Ki}. Hence, only finitely many of the \textit{different} logarithmic tracts $\T_F$ of $F$, that is, up to their $2\pi i$-translates, intersect $\exp^{-1}(S_R)$, and these intersections lie in the vertical strip $V(\log L +8\pi, Q)$. Since each of these pieces of logarithmic tracts must be disjoint from their $2\pi i$-translates, there exists an upper bound for the Euclidean diameter of any of them, say $\Delta>0$. Thus, by \eqref{eq_commute_log}, $F^n(w)$ and $F^n(w_\eta)$ belong to the same one of these pieces of tracts, and hence, $\vert F^n(w)-F^n(w_\eta)\vert \leq \Delta$. By this and Proposition \ref{prop_contraction}, 
\begin{equation}\label{eq_contractionG}
\vert w -w_\eta \vert \leq \frac{1}{2^n} \vert F^n(w)-F^n(w_\eta)\vert \leq \frac{\Delta}{2^n}<\Delta.
\end{equation}
Hence, $w_\eta \in \overline{V(\Rea w -\Delta, \Rea w+ \Delta )}$, and if $M\defeq \exp(\Delta)$, then 
$ z_\eta \in \overline{A(M^{-1}\vert z \vert, M \vert z \vert)}.$
\end{proof}
\section{The class $\CB$}\label{sec_CB}
Recall from the introduction that the class $\CB$ comprises all $f\in \B$ for which for $\lambda\in \C^\ast$ with small enough modulus, $J(\lambda f)$ is a Cantor bouquet. We start by studying the basic properties of functions in $\CB$.

Two maps $f,g\in \B$ are \textit{quasiconformally equivalent ($\sim$ near infinity)} if there exist quasiconformal maps $\phi, \psi: \C \rightarrow \C$ such that 
\begin{equation}\label{eq_qcfequiv}
\psi(g(z))=f(\phi(z))
\end{equation} 
for all $z\in \C$ ($\sim$ whenever $\vert g(z) \vert$ or $\vert f(\phi(z)) \vert$ is large enough). 
Moreover, for an entire function with bounded singular set, its \textit{parameter space} is the collection of all entire functions quasiconformally equivalent to it.

\begin{observation}[All functions in $\B$ have disjoint type maps in their parameter space] \label{prop_lambdadisjoint} Let $f\in \B$. Then, for all $\lambda\in \C^\ast$ with small enough modulus, the maps $h_\lambda\defeq \lambda f$ and $g_\lambda$ given by $g_\lambda(z)\defeq f(\lambda z)$ are of disjoint type and belong to the parameter space of $f$.
\end{observation}
\begin{proof}
The maps $h_\lambda$ and $g_\lambda$ are trivially quasiconformally equivalent to $f$; \eqref{eq_qcfequiv} holds for $g_\lambda$ by taking $\psi$ to be the identity map and $\phi$  as $z \mapsto \lambda z$, and for $h_\lambda$, \eqref{eq_qcfequiv} holds taking $\psi$ to be $z \mapsto z/\lambda$ and $\phi$ the identity map. Since $f \in \mathcal{B}$, we can choose $R>0$ such that $\lbrace S(f),0, f(0)\rbrace \subset \D_R$. For $\lambda\in \C$ with $\vert \lambda \vert$ sufficiently small, $f(\lambda\D_R) \subset \D_R$ and $\lambda f(\D_R)\subset \D_R$. Thus, for any such $\lambda$, $g_{\lambda}(\D_R)\subset \D_R$ and $h_{\lambda}(\D_R)\subset \D_R$. Moreover, for every $\lambda\in\C^{*}$, it is easy to see that $S(h_{\lambda})=\lambda S(f) \subset \D_R$ and $S(g_{\lambda})=S(f) \subset \D_R$. The combination of these two facts is enough to characterize 
$g_\lambda$ and $h_\lambda$ as disjoint type maps, see Proposition \ref{prop_disjoint}.
\end{proof}

Note that if two maps are quasiconformally equivalent to a third one, then they are quasiconformally equivalent to each other, since the composition of two quasiconformal maps is quasiconformal. In particular, all disjoint type maps in the parameter space of a function $f\in \B$ are pairwise quasiconformally equivalent, and the following result tells us that their dynamics are closely related.
\begin{thm}[Conjugacy between disjoint type maps {\cite[Theorem 3.1]{lasseRigidity}}]\label{thm_lasseR31} Any two quasiconformally equivalent disjoint type maps are conjugate on their Julia sets.
\end{thm}

In particular, a conjugacy between Julia sets implies that the topological structure of the sets must be the same. Even if in general, this does not necessarily imply that their embeddings in the plane are the same (i.e., they might not be ambiently homeomorphic), it follows from the proof of Theorem \ref{thm_lasseR31} that the map that conjugates two functions as in Theorem \ref{thm_lasseR31} is an ambient homeomorphism.
\begin{cor}\label{cor_allJg_CB} Let $f\in \B$. Then any two disjoint type maps on its parameter space are conjugate on their Julia sets, and hence these sets are homeomorphic. 
\end{cor}
This corollary implies that in order to determine if a function $f\in \B$ belongs to $\CB$, it is enough to check the topological structure of the Julia set of any disjoint type map quasiconformally equivalent to it:

\begin{prop}[Class $\CB$]\label{prop_classCB} An entire function $f\in \B$ belongs to $\CB$ if and only if the Julia set of any (and thus all) disjoint type function on its parameter space is a Cantor bouquet.
\end{prop}
Many functions in $\CB$ have been studied before, e.g., \cite{dierkCosine, Schleicher_Zimmer_exp, Baranski_Trees,RRRS, lasseBrushing}. To illustrate this, we gather together, up to date and to the author's knowledge, classes of functions appearing on the literature that are known to be in~$\CB$. For the definition of a logarithmic transform satisfying a \textit{head-start condition}, see {\cite[Definition~4.1]{RRRS}}.

\begin{prop}[Sufficient conditions for functions in $\CB$]\label{prop_resultsCB} A transcendental entire function $f\in \B$ is in class $\CB$ if one of the following holds:
\begin{enumerate}[label=(\Alph*)]
\item \label{itemA_exCB} $f$ is a finite composition of functions of finite order,
\item \label{itemB_exCB} there exists a logarithmic transform $F$ of $f$ that satisfies a linear head-start condition.
\item \label{itemC_exCB} there exists a logarithmic transform $G$ of a disjoint type function $g(z) \defeq f(\lambda z)$ for $\lambda \in \C ^\ast$ small enough, such that $G$ satisfies a uniform head-start condition.
\end{enumerate}
\end{prop}
\begin{proof}
If $f\in \B$ is a finite composition of functions of finite order, then by \cite[Theorem~5.6 and Lemma~5.7]{RRRS}, there exists a logarithmic transform $F\colon \T_F \rightarrow \H_{\log L}$ of $f$ satisfying a linear head-start condition for some $\phi$. In this case, the map
$$G: \T_G=(\T_F-\log \lambda)\rightarrow \H_{\log L} \quad \text{ given by } \quad G(w)\defeq F(w +\log\lambda)$$
is a logarithmic transform of the map $g$ given by $g(z)\defeq f(\lambda z)$. This is because we have $ \e^{F(w +\log\lambda)}=f(\e^{w +\log\lambda})=f(\lambda\e^{w})=g(\e^{w})=\e^{G(w)}$. Moreover, since $F$ satisfies a linear head-start condition, by its definition, $G$ also satisfies a linear (in particular uniform) head-start condition, and for $\lambda$ sufficiently small, by Observation \ref{prop_lambdadisjoint} and Proposition \ref{prop_disjoint}, $G$ and $g$ are of disjoint type. By \cite[Corollary 6.3]{lasseBrushing}, since $G$ satisfies a uniform head-start condition and is of disjoint type, both $J(G)$ and $J(g)$ are Cantor bouquets. We have shown that $\ref{itemA_exCB}\Rightarrow \ref{itemB_exCB} \Rightarrow \ref{itemC_exCB}\Rightarrow f \in \CB$, and so the statement follows.
\end{proof}

Our next goal is to show that  class $\CB$ is closed under iteration. This will be a consequence of Proposition \ref{prop_qcinfty} on composition of maps that are quasiconformally equivalent near infinity. We use in our proof some tools from quasiconformal maps gathered in \cite[Section 2]{lasseRigidity}. We refer to \cite{lehto_quasiconformal, vuorinen2006conformal} for definitions. 
In particular, we will make use of the following auxiliary result regarding interpolation of quasiconformal maps on annuli. The original source is \cite{lehto}, but we present in the next proposition the reformulation given in \cite[Proposition 2.11]{lasseRigidity}.
\begin{prop}[Interpolation of quasiconformal maps on annuli {\cite{lehto}}]\label{prop_letho} Let $A,B\subset \C$ be two bounded annuli, each bounded by two Jordan curves. Suppose that $\psi, \varphi \colon\C \rightarrow \C$ are quasiconformal maps such that $\psi$ maps the inner boundary $\alpha^-$ of $A$ to the inner boundary $\beta^-$ of $B$, and $\varphi$ takes the outer boundary $\alpha^+$ of $A$ to the outer boundary $\beta^+$ of $B$. Then there is a quasiconformal map $\tilde{\varphi} \colon\C \rightarrow \C$ that agrees with $\psi$ on the bounded component of $\C\setminus A$ and with $\varphi$ on the unbounded component of $\C\setminus A$.
\end{prop}

\begin{prop}[Composition of quasiconformally equivalent maps]\label{prop_qcinfty} Let $f_1,g_1\in\B$ be quasiconformally equivalent near infinity to $f_2$ and $g_2$, respectively. Then, $f_1\circ g_1$ is quasiconformally equivalent near infinity to $f_2 \circ g_2$.
\end{prop}
\begin{proof}
By assumption, there exist quasiconformal maps
$\phi_f,\phi_g, \psi_g, \psi_f: \C \rightarrow \C$ such that 
\begin{equation}\label{eq_defqcinf}
\psi_f(f_1(z))=f_2(\phi_f(z)) \qquad \text{ and } \qquad \psi_g(g_1(z))=g_2(\phi_g(z)) 
\end{equation} 
whenever $\max \{\vert f_1(z)\vert , \vert f_2(\phi_f(z)) \vert\}>R_f $ and $\max \{\vert g_1(z)\vert , \vert g_2(\phi_g(z)) \vert\}>R_g $ for some fixed $R_f, R_g>0.$ 
Equivalently, the semiconjugacies between the functions $f_1$ and $f_2$, and $g_1$ and $g_2$, are respectively defined in the sets
\begin{equation}\label{a}
\begin{split}
A(R_f) &\defeq f_1^{-1}( \C \setminus \D_{R_f}) \cup \phi_f^{-1}(f_2^{-1}( \C \setminus \D_{R_f})) \quad \text{ and } \\
B(R_g)&\defeq g_1^{-1}( \C \setminus \D_{R_g}) \cup \phi_g^{-1}(g_2^{-1}( \C \setminus \D_{R_g})).
\end{split}
\end{equation}
By increasing the constant $R_f$, we can assume that $A(R_f)\Subset \C \setminus (\D_{R_g} \cup \psi^{-1}_g(\D_{R_g}))$ and that there exists an annulus $\mathcal{R}\subset \C$ such that $\D_{R_g} \cup \psi^{-1}_g(\D_{R_g})$ is compactly contained in the bounded component of $\C \setminus \overline{\mathcal{R}}$, $A(R_f)$ in the unbounded one, and such that if $\mathcal{R}^{-}$ and $\mathcal{R}^{+}$ are the inner and outer boundaries of $\mathcal{R}$, then the curves $\psi_g(\mathcal{R}^{-})$ and $\phi_f(\mathcal{R}^{+})$ are the respective inner and outer boundaries of a topological annulus; see Figure \ref{fig: quasiconformally_equiv}.
\begin{figure}[htb] 

\begingroup%
  \makeatletter%
  \providecommand\color[2][]{%
    \errmessage{(Inkscape) Color is used for the text in Inkscape, but the package 'color.sty' is not loaded}%
    \renewcommand\color[2][]{}%
  }%
  \providecommand\transparent[1]{%
    \errmessage{(Inkscape) Transparency is used (non-zero) for the text in Inkscape, but the package 'transparent.sty' is not loaded}%
    \renewcommand\transparent[1]{}%
  }%
  \providecommand\rotatebox[2]{#2}%
  \newcommand*\fsize{\dimexpr\f@size pt\relax}%
  \newcommand*\lineheight[1]{\fontsize{\fsize}{#1\fsize}\selectfont}%
  \ifx\svgwidth\undefined%
    \setlength{\unitlength}{425.19685039bp}%
    \ifx\svgscale\undefined%
      \relax%
    \else%
      \setlength{\unitlength}{\unitlength * \real{\svgscale}}%
    \fi%
  \else%
    \setlength{\unitlength}{\svgwidth}%
  \fi%
  \global\let\svgwidth\undefined%
  \global\let\svgscale\undefined%
  \makeatother%
  \begin{picture}(1,0.63333333)%
    \lineheight{1}%
    \setlength\tabcolsep{0pt}%
    \put(0,0){\includegraphics[width=\unitlength,page=1]{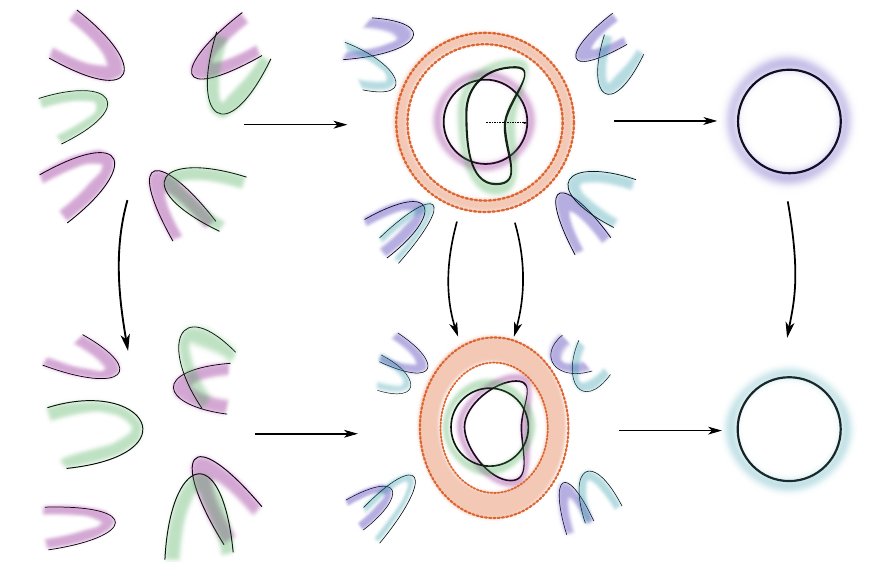}}%
    \put(0.60098262,0.31930562){\color[rgb]{0,0,0}\makebox(0,0)[lt]{\lineheight{1.25}\smash{\begin{tabular}[t]{l}$\varphi_f$\end{tabular}}}}%
    \put(0.45974742,0.31874461){\color[rgb]{0,0,0}\makebox(0,0)[lt]{\lineheight{1.25}\smash{\begin{tabular}[t]{l}$\psi_g$\end{tabular}}}}%
    \put(0.72852061,0.50844259){\color[rgb]{0,0,0}\makebox(0,0)[lt]{\lineheight{1.25}\smash{\begin{tabular}[t]{l}$f_1$\end{tabular}}}}%
    \put(0.47674566,0.60022756){\color[rgb]{0.8627451,0.37647059,0.16470588}\makebox(0,0)[lt]{\lineheight{1.25}\smash{\begin{tabular}[t]{l}$\mathcal{R}$\end{tabular}}}}%
    \put(0.32585215,0.15473022){\color[rgb]{0,0,0}\makebox(0,0)[lt]{\lineheight{1.25}\smash{\begin{tabular}[t]{l}$g_2$\end{tabular}}}}%
    \put(0.73060567,0.15763428){\color[rgb]{0,0,0}\makebox(0,0)[lt]{\lineheight{1.25}\smash{\begin{tabular}[t]{l}$f_2$\end{tabular}}}}%
    \put(0.0921122,0.32247884){\color[rgb]{0,0,0}\makebox(0,0)[lt]{\lineheight{1.25}\smash{\begin{tabular}[t]{l}$\varphi_g$\end{tabular}}}}%
    \put(0.90662791,0.33052338){\color[rgb]{0,0,0}\makebox(0,0)[lt]{\lineheight{1.25}\smash{\begin{tabular}[t]{l}$\psi_f$\end{tabular}}}}%
    \put(0.56900314,0.00099915){\color[rgb]{0.85882353,0.37647059,0.16470588}\makebox(0,0)[lt]{\lineheight{1.25}\smash{\begin{tabular}[t]{l}$\varphi_f(\mathcal{R}^+)$\end{tabular}}}}%
    \put(0.4228242,0.00526306){\color[rgb]{0.85882353,0.37647059,0.16470588}\makebox(0,0)[lt]{\lineheight{1.25}\smash{\begin{tabular}[t]{l}$\psi_g(\mathcal{R}^-)$\end{tabular}}}}%
    \put(0,0){\includegraphics[width=\unitlength,page=2]{qcequivinf.pdf}}%
    \put(0.14017185,0.62320618){\color[rgb]{0,0,0}\makebox(0,0)[lt]{\lineheight{1.25}\smash{\begin{tabular}[t]{l}$B(R_g)$\end{tabular}}}}%
    \put(0,0){\includegraphics[width=\unitlength,page=3]{qcequivinf.pdf}}%
    \put(0.57543742,0.62241292){\color[rgb]{0,0,0}\makebox(0,0)[lt]{\lineheight{1.25}\smash{\begin{tabular}[t]{l}$A(R_f)$\end{tabular}}}}%
    \put(0,0){\includegraphics[width=\unitlength,page=4]{qcequivinf.pdf}}%
    \put(0.29888786,0.50115977){\color[rgb]{0,0,0}\makebox(0,0)[lt]{\lineheight{1.25}\smash{\begin{tabular}[t]{l}$g_1$\end{tabular}}}}%
    \put(0,0){\includegraphics[width=\unitlength,page=5]{qcequivinf.pdf}}%
    \put(0.556,0.31947917){\color[rgb]{0,0,0}\makebox(0,0)[lt]{\lineheight{1.25}\smash{\begin{tabular}[t]{l}$\tilde{\psi}_g$\end{tabular}}}}%
    \put(0,0){\includegraphics[width=\unitlength,page=6]{qcequivinf.pdf}}%
    \put(0.19788461,0.32212474){\color[rgb]{0,0,0}\makebox(0,0)[lt]{\lineheight{1.25}\smash{\begin{tabular}[t]{l}$\tilde{\varphi}_g$\end{tabular}}}}%
    \put(0.53984044,0.50050517){\color[rgb]{0,0,0}\makebox(0,0)[lt]{\lineheight{1.25}\smash{\begin{tabular}[t]{l}$R_g$\end{tabular}}}}%
    \put(0.54336304,0.15564789){\color[rgb]{0,0,0}\makebox(0,0)[lt]{\lineheight{1.25}\smash{\begin{tabular}[t]{l}$R_g$\end{tabular}}}}%
    \put(0.89054779,0.15480815){\color[rgb]{0,0,0}\makebox(0,0)[lt]{\lineheight{1.25}\smash{\begin{tabular}[t]{l}$R_f$\end{tabular}}}}%
    \put(0.8906839,0.50191185){\color[rgb]{0,0,0}\makebox(0,0)[lt]{\lineheight{1.25}\smash{\begin{tabular}[t]{l}$R_f$\end{tabular}}}}%
  \end{picture}%
 \endgroup

\caption{Proof of Proposition \ref{prop_qcinfty} by interpolating the maps $\psi_g$ and $\phi_f$ using the annulus $\mathcal{R}$ shown in orange.}
\label{fig: quasiconformally_equiv}
\end{figure} 
	
By Proposition \ref{prop_letho}, we can interpolate $\psi_g$ and $\phi_f$ the following way: there exists a quasiconformal map $\tilde{\psi}_g :\C \rightarrow \C$ that agrees with $\psi_g$ in the bounded component of $\C\setminus \mathcal{R}$, and with $\phi_f$ in the unbounded component of $\C\setminus \mathcal{R}$. In particular, since by construction $\tilde{\psi}_g \equiv \psi_g$ in $(\partial \D_{R_g} \cup \psi^{-1}_g(\D_{R_g}))$, and by the Alexander trick, the isotopy class of a homeomorphism between two Jordan domains is determined by its boundary values,
\begin{equation}\label{eq_isotopy}
\tilde{\psi}_g \text{ is isotopic to } \psi_g \text{ in } X\defeq \C \setminus (\D_{R_g} \cup \psi^{-1}_g(\D_{R_g})) \text{ relative to } \partial X.
\end{equation}
	
Our next goal is to replace the map $\phi_g$ by another map $\tilde{\phi}_g$ by ``\emph{lifting}'' the map $\tilde{\psi}_g$ in such a way that the relation $\tilde{\psi}_g(g_1(z))=g_2(\tilde{\phi}_g(z))$ holds for all $z \in B(R_g)$. In order to do so, note that by construction, the restriction of $\phi_g$ to any connected component $C \subset B(R_g)$ is a homeomorphism into its image, and moreover, the closure of each domain $\phi_g(C)$ is by definition mapped injectively to $\C \setminus (\D_{R_g} \cup \psi_g(\D_{R_g}))$ under $g_2$. Hence, the inverse branches of $g_2$ from $\C \setminus (\D_{R_g} \cup \psi_g(\D_{R_g}))$ to each $\phi_g(C)$, which we denote by $g^{-1}_2\vert_{\phi_g(C) }$, are well-defined, and by \eqref{eq_defqcinf}, $\phi_g(C)=(g^{-1}_2\vert_{\phi_g(C)}\circ\psi_g\circ g_1)(C)$. We define, from each connected component $C$ of $B(R_g)$, a map $\tilde{\phi}_{g}\vert_{C}: C\rightarrow \phi_g(C)$ as
$$\tilde{\phi}_{g}\vert_{C} \defeq g^{-1}_2\vert_{\phi_g(C) }\circ \tilde{\psi}_g \circ g_1.$$
By \eqref{eq_isotopy} and since $\tilde{\psi}_g \equiv \psi_g$ in $g_1(\partial C)$, we have that $\phi_g$ and $\tilde{\phi}_{g}\vert_{C}$ are isotopic in $C$. Moreover, by construction, the continuous extension of $\tilde{\phi}_{g}\vert_{C} $ to $\partial C$ equals $\phi_{g}\vert_{\partial C}$. Thus, we can define a homeomorphism $\tilde{\phi}_g: \C \rightarrow \C$ as 
\begin{align*}
\tilde{\phi}_g(z)&\defeq \renewcommand{\arraystretch}{1.5}\left\{\begin{array}{@{}l@{\quad}l@{}} 
\tilde{\phi}_{g}\vert_{C} & \text{if} \quad z\in C \text{ for some } C \subset B(R_g);\\
\phi_g(z) & \text{otherwise.} \\
\end{array}\right.\kern-\nulldelimiterspace
\end{align*}	
By the glueing lemma, see \cite[Proposition 2.10]{lasseRigidity}, $\tilde{\phi}_g$ is a quasiconformal map. Consequently, $\tilde{\psi}_g(g_1(z))=g_2(\tilde{\phi}_g(z))$ for all $z\in B(R_g)$. Since, by construction, $\phi_f\equiv \tilde{\psi}_g$ in $A(R_f)$, we have that 
$$ (\psi_f \circ f_1\circ g_1)(z) = (f_2 \circ\phi_f \circ g_1)(z) = (f_2 \circ \tilde{\psi}_g \circ g_1)(z) = (f_2 \circ g_2 \circ \tilde{\phi}_g)(z)$$
for all $z\in \C$ such that $\max \{\vert f_1(g_1(z))\vert , \vert f_2 ( g_2 ( \tilde{\phi}_g(z)) \vert\}>R_f, $ as required.
\end{proof}

\begin{cor}[The class $\CB$ is closed under iteration]\label{cor_closediter}
If $f \in \CB$, then $f^{n} \in \CB$ for all $n\geq 0.$
\end{cor}
\begin{proof}
By Observation \ref{prop_lambdadisjoint}, we can choose $\lambda \in \C^\ast$ such that $h_\lambda \defeq \lambda f$ is a disjoint type function quasiconformally equivalent to $f$. Let us fix any $n>0$. By Proposition \ref{prop_qcinfty} used recursively $n-1$ times, the function $f^n$ is quasiconformally equivalent near infinity to $h_\lambda^n$. Moreover, again by Observation \ref{prop_lambdadisjoint}, we can choose $\mu \in \C^\ast$ such that $g_\mu \defeq \mu f^n$ is a disjoint type function quasiconformally equivalent to $f^n$. Then, $g_\mu$ and $h_\lambda^n$ are two disjoint type functions quasiconformally equivalent near infinity, and hence, they are quasiconformally conjugate on a neighbourhood of infinity \cite[Theorem 1.4]{lasseRigidity}. Moreover, by taking $\mu$ and $\lambda$ with sufficiently small modulus, one can see that that the conjugacy from \cite{lasseRigidity} extends to a neighbourhood of their Julia sets. In particular, by assumption, $J(h_\lambda)$ is a Cantor bouquet, and since Julia sets are completely invariant, $J(h_\lambda^n)$, and thus $J(g_\mu)$, are also Cantor bouquets. Consequently, we have shown that $f^n \in \CB$.
\end{proof}

\subsection*{Conjugacy near infinity}

In order to prove Theorems \ref{thm_intro_rays} and \ref{thm_intro_CB}, we shall use a general result from \cite{lasseRigidity}, that can be regarded as a sort of analogue of Böttcher's Theorem for functions in class $\B$. More specifically, that result allows us to conjugate the dynamics \textit{near infinity} of any function $f\in \CB$ to those of a disjoint type function in its parameter space, see Corollary \ref{cor_ridigity}. In particular, since by assumption any such disjoint type map has a Cantor bouquet Julia set, this conjugacy will imply the existence of a strongly absorbing Cantor bouquet in $J(f)$.

\begin{discussion}\label{discussion_nearinfty}
Let us fix $f\in \CB$ and choose any $K>0$ such that $S(f)\subset \D_K$. Let $L\geq K$ be any constant sufficiently large such that $f(\D_{K})\subset \D_L$. In particular, no preimage of $\C \setminus \D_L$ intersects $\D_{K}$, that is,
\begin{equation} \label{eq_defdisj}
f^{-1}\left(\C \setminus \D_L\right) \subset \C \setminus \D_{K}.
\end{equation}
Denote
\begin{equation} \label{eq_lambda}
\lambda\defeq \dfrac{K}{e^{8\pi}L} \quad \text{ and define} \quad g\defeq g_{\lambda} \colon \C \rightarrow \C \quad \text{ as }\quad g_\lambda(z)\defeq f(\lambda z).
\end{equation}
Then, for each $z\in \C$ such that $\vert g(z)\vert=\vert f(\lambda z)\vert>L$, by \eqref{eq_defdisj}, $\vert\lambda z\vert>K$, and so  $\vert z\vert>e^{8\pi}L$. That is, 
\begin{equation} \label{eq_disjt}
g^{-1}\left(\C \setminus \D_L\right) \subset \C \setminus \D_{\e^{8\pi}L}.
\end{equation}
By Proposition \ref{prop_disjoint} and Observation \ref{prop_lambdadisjoint}, $g$ is a disjoint type map in the parameter space of $f$. Let $\T_f$ be the set of \emph{tracts} of $f$, defined as the connected components of $f^{-1}\left(\C \setminus \D_L\right)$. Next, we fix logarithmic transforms for $f$ and $g$: let
\begin{equation*}
F: \mathcal{T}_F \rightarrow \H_{\log L}, \quad \ \text{with } \ \T_F\defeq \exp^{-1}(\T_f)\quad \text{and }\quad \H_{\log L}\defeq \exp^{-1}(\C \setminus \overline{\D_{L}}).
\end{equation*}
Note that by \eqref{eq_defdisj}, we have $\T_F \subset \H_{\log K}$. Then, the map
$$G\defeq G_{\lambda}: \T_G=(\T_F-\log \lambda)\rightarrow \H_{\log L} \quad \text{ given by } \quad G(w)\defeq F(w +\log\lambda)$$
is a logarithmic transform for $g$, since $ \e^{F(w +\log\lambda)}=f(\e^{w +\log\lambda})=f(\lambda\e^{w})=g(\e^{w})=\e^{G(w)}$. By the choice of the constant $\lambda$ in \eqref{eq_lambda}, it holds that 
\begin{equation}\label{eq_logG}
\T_{G} \subset \H_{\log L+8\pi}.
\end{equation}
\end{discussion}
We can now state the aforementioned result from \cite{lasseRigidity}, for the logarithmic transforms $F$ and $G$ just defined. For any logarithmic transform $F$ and constant $Q>0$, we denote
$$J_Q(F) \defeq \lbrace z \in J(F) : \Rea(F^{n}(z))\geq Q \text{ for all } n \geq 1\rbrace.$$
The following result is a compendium of \cite[Theorem 3.2, Lemma 3.3 and Theorem~3.4]{lasseRigidity}, in a version adapted to our setting\footnote{\cite[p. 250]{lasseRigidity} with $F_0=G$, $\kappa=-\log \lambda$, and $F_\kappa=F$.}:
\begin{thm}[Conjugacy near infinity for logarithmic transforms]\label{thm_ridigity} 
Let $\lambda$ be the constant, and let $F$ and $G$ be the logarithmic transforms, fixed in \ref{discussion_nearinfty}. For every $Q> 2\vert\log \lambda\vert+1$, there is a continuous map\footnote{The map $\Theta$ extends to a quasiconformal map $\Theta: \C \rightarrow \C$; see \cite[Theorem 3.4 or Theorem 1.1]{lasseRigidity}.} $\Theta\defeq \Theta^{\lambda} \colon J_Q(G) \rightarrow J(F)$ such that 
\begin{equation}\label{eq_ridigity}
\Theta \circ G= F \circ \Theta, \quad \qquad |\Theta(w)-w|\leq 2\vert\log\lambda\vert 
\end{equation}
and is a homeomorphism onto its image. Moreover, $J_{2Q}(F) \subset \Theta(J_Q(G))$ and $\Theta$ can be chosen so that $\Theta(w+2\pi i)=\Theta(w)+2\pi i$.
\end{thm}
\begin{remark}
In \cite[Theorem 3.2]{lasseRigidity}, it is assumed that the logarithmic transforms $F$ and $G$ are \textit{normalized}. Even if this might not be the case for the transforms $F$ and $G$ fixed in \ref{discussion_nearinfty}, all that is required in the proof of \cite[Theorem 3.2]{lasseRigidity} is that \eqref{eq_logG} holds, and thus the theorem applies in our setting.
\end{remark}
We can transfer this result to the dynamical planes of $f$ and $g$. Recall from \eqref{eq_JR} that for each entire function $f$ and constant $R\geq 0$, we denote
\[J_R(f) \defeq \lbrace z \in J(f) : |f^n(z)|\geq R \text{ for all } n \geq 1\rbrace \quad \text{ and } \quad I_R(f)\defeq I(f) \cap J_R(f).\]
\begin{cor}[Conjugacy near infinity]\label{cor_ridigity} Let $\lambda$ be the constant, and let $f$ and $g$ be the functions, fixed in \ref{discussion_nearinfty}. Then, for every $R>\exp(2\vert\log \lambda\vert +8\pi +\log L)$, there exists a continuous map $\theta\defeq \theta_R \colon J_R(g)\rightarrow J(f) \cap \mathcal{T}_f$ such that 
\begin{equation}\label{eq_commute_cor}
\theta \circ g =f\circ \theta
\end{equation}
and is a homeomorphism onto its image. Moreover, $J_{e^2R}(f) \subset \theta(J_R(g))$, and for every point $z \in J_R(g)$, $\theta(z) \in \overline{A(\lambda^{2}\vert z\vert, \lambda^{-2}\vert z\vert)}$. In particular, $\theta(I_R(g))\subset I(f)$. 
\end{cor}

\begin{proof}
Let us fix any constant $R$ as in the statement, and let $Q\defeq \log(R)$. Note that by assumption, $Q> 2\vert \log \lambda \vert +8\pi +\log L$, and hence, we can apply Theorem \ref{thm_ridigity} to the logarithmic transforms $F$ and $G$ defined in \ref{discussion_nearinfty}. In particular, by \eqref{eq_ridigity} and the lower bound on $Q$, it holds that $$\Theta(J_Q(G))\subset \H_{Q-2\vert \log \lambda \vert}\subset \H_{\log L + 8\pi} \cap J(F).$$
In addition, by the commutative relation in \eqref{eq_ridigity}, $F(\Theta(J_Q(G))) \subset \Theta(J_Q(G))\subset J(F)$. Hence, since $F$ is a logarithmic transform of $f$, by Observation \ref{rem_juliaLog} applied to the set $\Theta(J_Q(G))$, it holds that $\exp(\Theta(J_Q(G)))\subset J(f)\cap \mathcal{T}_f$. Moreover, since $g$ is of disjoint type, by \eqref{eq_Juliadisjoint}, $\exp(J(G))= J(g)$. By all of the above and since the map $\Theta$ is $2\pi i$-periodic, there exists a map $\theta \colon J_R(g) \rightarrow J(f)$ defined by the relation $\exp \circ \Theta =\theta\circ\exp$. Then, 
$$\theta \circ g \circ \exp =\theta \circ \exp \circ G =\exp\circ \Theta\circ G=\exp \circ F \circ \Theta=f \circ \exp \circ \Theta= f \circ \theta \circ \exp,$$ 
and since $\exp$ is a continuous surjective map, \eqref{eq_commute_cor} holds. This is depicted in the following diagram:
\[ \begin{tikzcd}
J_Q(G) \arrow[d,"\Theta" '] \arrow["G" ,rrr, bend left =22] \arrow[r, "\exp"] & {\color{blue} J_R(g)} \arrow[r, blue, "g"] \arrow[d, blue, "\theta"']& {\color{blue} J_R(g)} \arrow[d, blue, "\theta"] & J_Q(G) \arrow[d,"\Theta"] \arrow[l, "\exp"'] \\
J(F) \arrow["F"' ,rrr, bend right =22] \arrow[r, "\exp"] &{\color{blue} J(f)} \arrow[r, blue, "f"] & {\color{blue} J(f)} & J(F) \arrow[l, "\exp"']. 
\end{tikzcd}
\]
By \eqref{eq_ridigity}, for any $w\in J_Q(G)$, 
$$w, \Theta(w),  \in \overline{V(\Rea(w) -2 \vert \log \lambda \vert, \Rea(w) +2 \vert \log \lambda \vert)}.$$
Hence, if $z \in J_R(g)$, then $z, \theta(z) \in \overline{A(\lambda^{2}\vert z \vert, \lambda^{-2}\vert z \vert)}$, and in particular, $\theta(I_R(g))\subset I(f)$.
\end{proof}

It is possible to define external addresses for any $f\in \B$ as equivalence classes. We follow \cite[p. 2107]{lasseBrushing}.
\begin{defn}[External addresses for $f\in \B$]
Let $f\in\B$ and let $F$ be a logarithmic transform of $f$. For $\s= T_0 T_1 T_2 \dots$ and $\underline{\tau} = T_0' T_1' T_2' \dots$ in $\Addr(J(F)),$
\begin{equation}\label{eq_simAddr}
\s \sim \underline{\tau} \quad \iff \quad T_0'=T_0+2\pi i k \text{ for some } k\in \Z \text{ and } T_j=T_j' \text{ for all }j>0.
\end{equation}
An \textit{external address} of $f$ is an equivalence class of $\Addr(J(F))/\sim.$
\end{defn}

We note that the definition of external addresses we have provided is equivalent to the more natural definition in terms of some simply connected unbounded domains called \textit{fundamental domains}, see e.g. \cite{lasse_dreadlocks, mio_thesis}. For details on this equivalence, see for example \cite{Benini_Fagella2017}.

\begin{remark}[Correspondence between addresses] \label{obs_corresp_addr}
Following Corollary \ref{cor_ridigity}, the map $\theta$ establishes a $1$-to-$1$ order-preserving correspondence between external addresses of $f$ and $g$. This follows from \cite[Proof of Theorem 3.2]{lasseRigidity}, where it is stated that the map $\Theta$ establishes such bijection between external addresses of their logarithmic transforms. For more details, see \cite[Proposition~4.47 and Observation~4.48]{mio_thesis}.
\end{remark}

We note that Corollary \ref{cor_ridigity} does not require the function $f$ to be in class $\CB$ but only in $\B$, and thus $J(g)$ being a Cantor bouquet has not been used yet. We shall do so in the next theorem, from which Theorems \ref{thm_intro_CB} and \ref{thm_intro_rays} follow, and that summarizes the properties of these maps required in the proofs of the results in \cite{mio_splitting}. 
\begin{thm}\label{thm_CB} Let $f\in \CB$ and let $K>0$ so that $S(f)\subset \D_K$. Then, there exists a disjoint type map $g$ that has a strongly absorbing Cantor bouquet $X\subset J(g)$, and a  continuous map $\hat{\theta} \colon J(g)\rightarrow J(f)$ with the following properties: 
\begin{enumerate}[label=(\alph*)]
\item \label{item:inclusions}$ \hat{\theta}(J(g)) \cup J(g)\subset \C\setminus \D_K$;
\item \label{item:commute}  $\hat{\theta} \circ g =f\circ \hat{\theta} $, except in a bounded set $B\subset \C\setminus X$;
\item \label{item:boundedC} there is a bounded set $C$ so that if $z\in B$, then the curve $[\hat{\theta}(g(z)), f(\hat{\theta}(z))]$ belongs to $C\cap f^{-1}(\C \setminus \D_K)$.
\item \label{item:homeomX} $\hat{\theta}\vert_X$ is a homeomorphism onto its image;
\item \label{item:imageBouquet}$\hat{\theta}(X)$ is a strongly absorbing Cantor bouquet in $J_K(f)$, and so $f$ is criniferous;
\item Each hair of $\hat{\theta}(X)$ is either a ray tail or a dynamic ray with its endpoint;\label{item:rays}
\item  \label{item:annulus} there is 
$M\defeq M(K)>0$ such that for every $z \in J(g)$, $\hat{\theta}(z) \in \overline{A(M^{-1}\vert z\vert, M\vert z\vert)}$; in particular, $\hat{\theta}(X\cap I(g))\subset I(f)$;
\item  \label{item:addr} the map $\hat{\theta}$ establishes an order-preserving one-to-one correspondence between external addresses of $g$ and $f$.
\end{enumerate}
\end{thm}
\begin{proof}
Let $f\in \CB$, let $K>0$ as in the statement, and let $g\defeq g_\lambda$ be the disjoint type map defined in \ref{discussion_nearinfty}. In particular, $g$ belongs to the parameter space of $f$, and so by Proposition~\ref{prop_classCB}, $J(g)$ is a Cantor bouquet. Let us fix some constant $R>\exp(2\vert\log \lambda\vert +8\pi +L)$. Then, by Theorem \ref{thm_pi_cont}, there exists a strongly absorbing Cantor bouquet $X\subset J(g)$ and a continuous surjective map $\pi\colon J(g)\to X$ such that $X\subset J_R(g)$. Let $\theta\colon J_R(g)\to J(f)$ be the map from Corollary~\ref{cor_ridigity}, and define
\begin{equation*}
\hat{\theta}\defeq \theta\circ\pi\colon J(g)\to J(f),
\end{equation*}
which is a continuous function, as it is the composition of two continuous maps.

Since by Theorem \ref{thm_pi_cont}\ref{item:pi_id}, $\pi\vert_X$ is the identity map, and by Corollary \ref{cor_ridigity}, $\theta\vert_X$ is a homeomorphism to $\theta(X)$, \ref{item:homeomX} follows and $\hat{\theta}(X)$ is a Cantor bouquet. Since $\theta \circ g =f\circ \theta$ in $J_R(g)\supset X$,  
$$f(\hat{\theta}(X))=(f\circ\theta\circ\pi)(X)=(\theta\circ g)(X)\subset \theta(X)=\hat{\theta}(X).$$
Moreover, since $\hat{\theta}(X)\subset \theta(J_R(g))\subset J(f)\cap \T_f\subset \C \setminus \D_K$, using that $\theta \circ g =f\circ \theta$ in $J_R(g)$ we get that $\hat{\theta}(X)\subset J_K(f)$. As $X$ is strongly absorbing, $J_Q(g)\subset X$ for all $Q>R$ sufficiently large. Let us fix any such $Q> e^2R$ and let $z \in J_{\lambda^{-2}Q}(f)$. Since $\vert\lambda\vert<1$, $z \in J_{e^{2}R}(f)$, and so, there exists a unique $w\in J_R(g)$ such that $\theta(w)=z$. Since for all $n\geq 0$, $\vert f^n(z) \vert=\vert \theta(g^n(w)) \vert>\lambda^{-2}Q$, by Corollary \ref{cor_ridigity}, $\vert g^n(w)\vert> Q$, and so $w\in J_Q(g)$. Thus, $J_{\lambda^{-2}Q}(f)\subset \theta(J_Q(g))\subset \hat{\theta}(X).$ We have then shown that $\hat{\theta}(X)$ is a strongly absorbing Cantor bouquet. By Corollary~\ref{cor_CB_criniferous}, $f$ must be criniferous, and so \ref{item:imageBouquet} is proved. Moreover, by Proposition \ref{prop_Ts}, each hair of $J(g)$, and so of $X$, is either a ray tail or a dynamic ray with its endpoint. Since, in $\theta \circ g =f\circ \theta$ in $J_R(g)\supset X$, \ref{item:rays} holds.

Recall that $g$ has been defined so that $g^{-1}\left(\C \setminus \D_L\right) \subset \C \setminus \D_{\e^{8\pi}L}$ for some $L>K$. In particular, since $J(g)=\bigcap_{n\geq 0} g^{-n}(\C \setminus \D_L)$, see Proposition \ref{prop_disjoint}, $J(g)\subset \C \setminus \D_L$. In addition, since by Corollary \ref{cor_ridigity}, for all $z \in J_R(g)$, $\theta(z) \in \overline{A(\lambda^{2}\vert z\vert, \lambda^{-2}\vert z\vert)}$, and $R > \lambda^{-2}L>K$,
$$\hat{\theta}(J(g))\subset \C\setminus \D_{\lambda^{-2}R}\subset \C \setminus \D_K. $$
Thus, we have shown \ref{item:inclusions}. Moreover, by Theorem~\ref{thm_pi_cont}, there exists $N>R$ such that $\pi(z) \in \overline{A(N^{-1}\vert z \vert, N \vert z \vert)}$ for all $z\in J(g).$ Then, \ref{item:annulus} holds letting $M\defeq \lambda^{-2}N$.

Let $B\subset J(g)$ be the bounded set of points $z$ for which  $g(\pi(z))\neq \pi(g(z))$ provided by Theorem \ref{thm_pi_cont}\ref{item:pi_SR}. In particular, \ref{item:commute} follows. Moreover, by this result and \eqref{eq_commute_cor}, if $z\in B$, then $$\theta([\pi(g(z)),g(\pi(z))])=[\hat{\theta}(g(z)), f(\hat{\theta}(z))]\subset \hat{\theta}(\overline{B})\eqdef C.$$
Since $\hat{\theta}(\overline{B})\subset \hat{\theta}(J(g))\subset \T_f\subset f^{-1}(\C\setminus \D_K)$, \ref{item:boundedC} is proved. Finally, \ref{item:addr} is a direct consequence of the remark after Corollary \ref{cor_ridigity}. 
\end{proof}

\begin{proof}[Proof of Theorem \ref{thm_intro_CB}] It is a direct consequence of Theorem \ref{thm_CB}\ref{item:imageBouquet}. 
\end{proof}

\begin{proof}[Proof of Theorem \ref{thm_intro_rays}] If $f\in \CB$, then by Theorem \ref{thm_CB}\ref{item:imageBouquet}, $f$ is criniferous. The rest of the statement follows from \cite[Proposition 2.3]{RRRS}.
\end{proof}

\bibliographystyle{alpha}
\bibliography{biblioComplex}

\begin{thebibliography}{{Par}19b}

\bibitem[AO93]{AartsOversteegen}
J.~Aarts and L.~Oversteegen.
\newblock The geometry of {J}ulia sets.
\newblock {\em Trans. Amer. Math. Soc.}, 338(2):897--918, 1993.

\bibitem[Bar07]{Baranski_Trees}
K.~Bara\'{n}ski.
\newblock Trees and hairs for some hyperbolic entire maps of finite order.
\newblock {\em Math. Z.}, 257(1):33--59, 2007.

\bibitem[BF20]{Benini_Fagella2017}
A.~M. Benini and N.~Fagella.
\newblock Singular values and non-repelling cycles for entire transcendental
  maps.
\newblock {\em Indiana Univ. Math. J.}, 69:1543--1558, 2020.

\bibitem[BJR12]{lasseBrushing}
K.~Bara{\'{n}}ski, X.~Jarque, and L.~Rempe.
\newblock Brushing the hairs of transcendental entire functions.
\newblock {\em Topology and its Applications}, 159(8):2102--2114, 2012.

\bibitem[BR20]{lasse_dreadlocks}
A.M. Benini and L.~{Rempe}.
\newblock A landing theorem for entire functions with bounded post-singular
  sets.
\newblock {\em Geom. Funct. Anal.}, 30:1465–1530, 2020.

\bibitem[DH84]{Orsaynotes}
A.~Douady and J.~H Hubbard.
\newblock {\em \'Etude dynamique des polynomes complexes}.
\newblock Orsay : {U}niversite de {P}aris-{S}ud, {D}ept. de {M}ath\'ematique,
  1984.

\bibitem[DK84]{devaney_Krych}
R.~L. Devaney and M.~Krych.
\newblock {Dynamics of $\exp(z)$}.
\newblock {\em Ergodic Theory and Dynamical Systems}, 4(1):35–52, 1984.

\bibitem[DT86]{devaney_tangerman}
Robert~L. Devaney and Folkert Tangerman.
\newblock Dynamics of entire functions near the essential singularity.
\newblock {\em Ergodic Theory Dynam. Systems}, 6(4):489--503, 1986.

\bibitem[EL92]{eremenkoclassB}
A.~{Erë}menko and M.~Lyubich.
\newblock Dynamical properties of some classes of entire functions.
\newblock {\em Ann. Inst. Fourier (Grenoble)}, 42(4):989--1020, 1992.

\bibitem[Er{\"{e}}89]{erem89}
A.~E. Er{\"{e}}menko.
\newblock On the iteration of entire functions.
\newblock {\em Banach Center Publications}, 23(1):339--345, 1989.

\bibitem[Fat26]{Fatou}
P.~Fatou.
\newblock Sur l'it\'{e}ration des fonctions transcendantes {E}nti\`eres.
\newblock {\em Acta Math.}, 47(4):337--370, 1926.

\bibitem[GK86]{goldberg_keen_86}
L.~R. Goldberg and L.~Keen.
\newblock A finiteness theorem for a dynamical class of entire functions.
\newblock {\em Ergodic Theory and Dynamical Systems}, 6(2):183–192, 1986.

\bibitem[Leh65]{lehto}
O.~Lehto.
\newblock An extension theorem for quasiconformal mappings.
\newblock {\em Proc. London Math. Soc.}, 14a:187–190, 1965.

\bibitem[LV73]{lehto_quasiconformal}
O.~Lehto and K.~Virtanen.
\newblock {\em Quasiconformal mappings in the plane \textit{(Second edition)}}.
\newblock Springer-Verlag, New York-Heidelberg, 1973.
\newblock Translated from the German by K. W. Lucas, Die Grundlehren der
  mathematischen Wissenschaften, Band 126.

\bibitem[{Mih}12]{helenaSemi}
H.~{Mihaljevi{\'c}-Brandt}.
\newblock Semiconjugacies, pinched {C}antor bouquets and hyperbolic orbifolds.
\newblock {\em Trans. Amer. Math. Soc.}, 364(8):4053--4083, 2012.

\bibitem[{Par}19a]{mio_thesis}
L.~{Pardo Sim\'on}.
\newblock Dynamics of transcendental entire functions with escaping singular
  orbits.
\newblock {\em PhD thesis, University of Liverpool}, 2019.

\bibitem[{Par}19b]{mio_splitting}
L.~{Pardo-Sim\'on}.
\newblock Splitting hairs with transcendental entire functions.
\newblock {\em Preprint, ArXiv:1905.03778v3}, 2019.

\bibitem[{Par}20]{mio_cosine}
L.~{Pardo-Sim\'on}.
\newblock Topological dynamics of cosine maps.
\newblock {\em Preprint, arXiv:2003.07250}, 2020.

\bibitem[{Par}21]{mio_orbifolds}
L.~{Pardo-Sim\'on}.
\newblock Orbifold expansion and entire functions with bounded {F}atou
  components.
\newblock {\em Ergodic Theory and Dynamical Systems}, page 1–40, 2021.

\bibitem[Pom92]{pommerenke_boundary}
Ch. Pommerenke.
\newblock {\em Boundary behaviour of conformal maps}, volume 299 of {\em
  Grundlehren der Mathematischen Wissenschaften [Fundamental Principles of
  Mathematical Sciences]}.
\newblock Springer-Verlag, Berlin, 1992.

\bibitem[Rem09]{lasseRigidity}
L.~Rempe.
\newblock Rigidity of escaping dynamics for transcendental entire functions.
\newblock {\em Acta Math.}, 203(2):235--267, 2009.

\bibitem[{Rem}16]{lasse_arclike}
L.~{Rempe-Gillen}.
\newblock Arc-like continua, {J}ulia sets of entire functions, and {E}remenko's
  {C}onjecture.
\newblock {\em Preprint, arXiv:1610.06278v3}, 2016.

\bibitem[RRRS11]{RRRS}
G.~Rottenfu{\ss}er, J.~R\"{u}ckert, L.~Rempe, and D.~Schleicher.
\newblock Dynamic rays of bounded-type entire functions.
\newblock {\em Annals of Mathematics (2)}, 173(1):77--125, 2011.

\bibitem[RS08]{dierkCosine}
G.~Rottenfu{\ss}er and D.~Schleicher.
\newblock Escaping points of the cosine family.
\newblock In Philip~J. Rippon and Gwyneth~M. Stallard, editors, {\em
  Transcendental Dynamics and Complex Analysis}, pages 396--424. Cambridge
  University Press, 2008.
\newblock Cambridge Books Online.

\bibitem[Sch93]{Schiff_normal}
Joel~L. Schiff.
\newblock {\em Normal families}.
\newblock Universitext. Springer-Verlag, New York, 1993.

\bibitem[Six18]{Dave_survey}
D.~J. Sixsmith.
\newblock Dynamics in the {E}remenko-{L}yubich class.
\newblock {\em Conform. Geom. Dyn.}, 22:185--224, 2018.

\bibitem[SZ03]{Schleicher_Zimmer_exp}
D.~Schleicher and J.~Zimmer.
\newblock Escaping points of exponential maps.
\newblock {\em Journal of the London Mathematical Society}, 67(2):380--400, 4
  2003.

\bibitem[Vuo88]{vuorinen2006conformal}
M.~Vuorinen.
\newblock {\em Conformal {G}eometry and Quasiregular Mappings}.
\newblock Lecture Notes in Mathematics, 1319. Springer-Verlag, Berlin, 1988.

\end{thebibliography}
\end{document}